\documentclass{article}
\usepackage{graphicx}
\usepackage{comment}
\usepackage{amsmath}
\usepackage{amsthm}
\usepackage{amssymb}
\usepackage[colorlinks=true,linkcolor=blue,citecolor=green]{hyperref}
\usepackage[margin=3cm]{geometry}
\usepackage{stmaryrd}
\usepackage{bussproofs}
\usepackage{mathtools}
\usepackage[all]{xy}
\usepackage{todonotes}
\usepackage{forest}
\forestset{smullyan tableaux/.style={for tree={math content},where n children=1{!1.before computing xy={l=\baselineskip},!1.no edge}{},closed/.style={label=below:$\times$},},}
\newcommand{\weakrightarrow}{\rightarrowtriangle} 
\newcommand{\weakcoimplies}{\multimap} 

\newcommand{\Ltriangle}{\mathcal{L}_\triangle}
\newcommand{\Ltrianglesquare}{\mathcal{L}^\neg_\triangle}
\newcommand{\NNF}{\mathsf{NNF}}
\newcommand{\Gtriangle}{\mathsf{G}\triangle}
\newcommand{\HGtriangle}{\mathcal{H}\Gtriangle}
\newcommand{\HGsquare}{\mathcal{H}\Gsquare}
\newcommand{\Prop}{\mathtt{Prop}}

\newcommand{\coimplies}{\Yleft}

\newcommand{\Gsquare}{\mathsf{G}^2}
\newcommand{\KGsquare}{\mathbf{K}\mathsf{G}^2}
\newcommand{\biG}{\mathsf{biG}}
\newcommand{\KbiG}{\mathbf{K}\mathsf{biG}}

\newcommand{\KG}{\mathfrak{GK}}
\newcommand{\HKbiG}{\mathcal{H}\KbiG}
\newcommand{\Sfconst}{\mathsf{Sf}^{\mathbf{0},\mathbf{1}}}

\newcommand{\Gsquareorder}{\mathsf{G}^2(\rightarrow,\coimplies)}
\newcommand{\HKGsquare}{\mathcal{H}\KGsquare}
\newcommand{\GsquareNelson}{\mathsf{G}^2(\weakrightarrow,\weakcoimplies)}
\newcommand{\Th}{\mathsf{Th}}
\newcommand{\MBox}{\mathsf{M}_{\Box}}
\newcommand{\HKGc}{\mathcal{H}\KG^c}
\newcommand{\HKG}{\mathcal{H}\KG}

\newcommand{\bimodalLtriangle}{\mathcal{L}_{\triangle,\Box,\lozenge}}
\newcommand{\bimodalLtrianglesquare}{\mathcal{L}^\neg_{\triangle,\Box,\lozenge}}

\newtheorem{theorem}{Theorem}[section]
\newtheorem{proposition}{Proposition}[section]

\newtheorem{lemma}{Lemma}[section]
\theoremstyle{definition}
\newtheorem{definition}{Definition}[section]
\theoremstyle{remark}
\newtheorem{remark}{Remark}[section]
\newtheorem{example}{Example}[section]
\newtheorem{convention}{Convention}[section]

\usepackage{longtable}
\usepackage{fancyhdr}
\usepackage{authblk}
\usepackage{soul}
\begin{document}
\providecommand{\keywords}[1]
{
  \small	
  \textbf{\textit{Keywords: }} #1
}
\title{Crisp bi-G\"{o}del modal logic and its paraconsistent expansion\thanks{The research of Marta B\'ilkov\'a was supported by the grant 22-01137S of the Czech Science Foundation. The research of Sabine Frittella and Daniil Kozhemiachenko was funded by the grant ANR JCJC 2019, project PRELAP (ANR-19-CE48-0006).\\We also thank the reviewer for their comments that greatly enhanced the quality of the paper.}}
\author[1]{Marta B\'{\i}lkov\'{a}}
\author[2]{Sabine Frittella}
\author[2]{Daniil Kozhemiachenko (corresponding author)}
\affil[1]{The Czech Academy of Sciences, Institute of Computer Science, Czech Republic}
\affil[2]{INSA Centre Val de Loire, Univ.\ Orl\'{e}ans, LIFO EA 4022, France}
\maketitle
\begin{abstract}
In this paper, we provide a Hilbert-style axiomatisation for the crisp bi-G\"{o}del modal logic $\KbiG$. We prove its completeness w.r.t.\ crisp Kripke models where formulas at each state are evaluated over the standard bi-G\"{o}del algebra on $[0,1]$. We also consider a paraconsistent expansion of $\KbiG$ with a De Morgan negation $\neg$ which we dub $\KGsquare$. We devise a Hilbert-style calculus for this logic and,  as a~con\-se\-quence of a~conservative translation from $\KbiG$ to $\KGsquare$, prove its completeness w.r.t.\ crisp Kripke models with two valuations over $[0,1]$ connected via $\neg$.

For these two logics, we establish that their decidability and validity are $\mathsf{PSPACE}$-complete.

We also study the semantical properties of $\KbiG$ and $\KGsquare$. In particular, we show that Glivenko's theorem holds only in finitely branching frames. We also explore the classes of formulas that define the same classes of frames both in $\mathbf{K}$ (the classical modal logic) and the crisp G\"{o}del modal logic $\KG^c$. We show that, among others, all Sahlqvist formulas and all formulas $\phi\rightarrow\chi$ where $\phi$ and $\chi$ are monotone, define the same classes of frames in $\mathbf{K}$ and $\KG^c$.

\end{abstract}
\section{Introduction}
The present paper is conceived as a natural continuation of the project commenced in~\cite{BilkovaFrittellaMajerNazari2020} and continued in~\cite{BilkovaFrittellaKozhemiachenko2021} and then in~\cite{BilkovaFrittellaKozhemiachenko2022IJCAR}. In the project, we aim at devising logics that formalise reasoning with inconsistent, incomplete and (or) uncertain information.

In the last paper, we provided two logics --- the bi-G\"{o}del modal logic $\KbiG$ (and its fuzzy version $\KbiG^\mathsf{f}$) and $\KGsquare$ --- its paraconsistent expansion with a De Morgan negation $\neg$. We also studied both logics whose validity was restricted to finitely branching frames and argued for their utility in the representation of agents' beliefs. In this paper, we aim at the study of the logics not restricted to the finitely branching frames.

\paragraph{G\"{o}del modal logics}
Conceptually, this paper has two sources of inspiration. First of all, we expand on the existing research on G\"{o}del modal logics as provided in multiple papers~\cite{CaicedoRodriguez2010,CaicedoMetcalfeRodriguezRogger2013,CaicedoRodriguez2015,CaicedoMetcalfeRodriguezRogger2017,RodriguezVidal2021}. G\"{o}del modal logics are well-researched: their complete axiomatisations (both over fuzzy and crisp Kripke frames) are established; their mono- and bi-modal fragments are shown to be decidable and, in fact, $\mathsf{PSPACE}$-complete; it is also known that they are strictly more expressive than the classical modal logic~$\mathbf{K}$.

On the other hand, the modal logics expanding the G\"{o}del logic with coimplication $\coimplies$ or the Baaz Delta $\triangle$\footnote{Note that $\coimplies$ and $\triangle$ are interdefinable:
$\triangle p\coloneqq\mathbf{1}\coimplies(\mathbf{1}\coimplies p)$ and $p\coimplies q\coloneqq p\wedge{\sim}\triangle(p\rightarrow q)$.} (also known as bi-G\"{o}del logic or symmetric G\"{o}del logic, $\biG$) have remained relatively unstudied. To the best of our knowledge, there are only two papers studying modal expansions of $\biG$. First is~\cite{GrigoliaKiseliovaOdisharia2016} where an algebraic semantics for the provability bi-G\"{o}del (symmetric G\"{o}del, in the authors' terminology) logic is studied. The other text is~\cite{AguileraDieguezFernandez-DuqueMcLean2022} where a linear temporal logic expanding $\biG$ is explored.

Furthermore, while it is well established that every class of frames classically definable by some formula $\phi(\overrightarrow{p_i})$ is also $\KG$-definable with $\phi(\overrightarrow{p_i}/\overrightarrow{{\sim\sim}p_i})$, there are examples of formulas that define the same class of frames both in $\mathbf{K}$ and the crisp G\"{o}del modal logic $\KG^c$: $\Box p\rightarrow p$ and $p\rightarrow\lozenge p$ define reflexive frames; $\lozenge\mathbf{1}$ defines serial frames, etc. However, there has been no systematic study which formulas can be ‘transferred’ in this manner from $\mathbf{K}$ to $\KG^c$. 

\paragraph{Paraconsistent expansions of the bi-intuitionistic logic}
The second source of inspiration is the study of the expansions of (super-)intuitionistic logics with the strong or constructive De Morgan negation as proposed in~\cite{Wansing2008}. In that paper, several constructive De Morgan negations for the bi-intuitionistic logic were studied: in particular, the Nelson negation that was initially proposed in~\cite{Nelson1949} and defined for the implication as $\neg(p\rightarrow q)\coloneqq p\wedge\neg q$, and the negation of the logic dubbed $\mathsf{I}_4\mathsf{C}_4$ by Wansing where $\neg(p\rightarrow q)$ is defined as $\neg q\coimplies\neg p$ and $\neg(p\coimplies q)\coloneqq\neg q\rightarrow\neg p$. The latter logic, in fact, was introduced several times\footnote{We are grateful to Heinrich Wansing for pointing this out to us.}: first by Moisil~\cite{Moisil1942} as symmetric propositional calculus, then by Wansing~\cite{Wansing2008} as $\mathsf{I}_4\mathsf{C}_4$, and then by Leitgeb~\cite{Leitgeb2019} as HYPE. Cf.~\cite{OdintsovWansing2021} for a recent and more detailed discussion. In~\cite{Wansing2008}, Nelson's logic with coimplication and $\mathsf{I}_4\mathsf{C}_4$ are equipped with frame semantics on bi-intuitionistic frames with two \emph{independent} valuations $e^+$ and $e^-$ that are interpreted as support of truth and support of falsity. The valuations are connected via the strong De Morgan negation in the following sense: support of falsity of $\phi$ is defined as support of truth of $\neg\phi$ and vice versa.

In~\cite{BilkovaFrittellaKozhemiachenko2021}, we discussed two paraconsistent logics collectively dubbed $\Gsquare$ expanding G\"{o}del logic with a De Morgan negation which were, in fact, pre-linear extensions of Nelson's logic $\mathsf{N4}$ and $\mathsf{I}_4\mathsf{C}_4$. We also provided them with algebraic semantics over the algebra $[0,1]^{\Join}$ --- defined on the twist product of the lattice $[0,1]$ with itself --- thus linking them to other paraconsistent fuzzy logics such as the ones in~\cite{ErtolaEstevaFlaminioGodoNoguera2015}. In~\cite{BilkovaFrittellaKozhemiachenkoMajer2022IJAR}, we applied $\Gsquareorder$ (the linear expansion of Moisil's logic) and $\GsquareNelson$ (the linear expansion of $\mathsf{N4}$) presented via Hilbert-style axiomatisations to study qualitative reasoning under uncertainty.

\paragraph{Logics}
In this paper, we will be discussing several logics obtained from the propositional G\"{o}del logic $\mathsf{G}$. Our main interest lies in the logics we denote $\KbiG$ and $\KGsquare$. They can be produced from $\mathsf{G}$ in several ways: (1) adding De Morgan negation $\neg$ to obtain $\Gsquare$ (in which case $\phi\coimplies\phi'$ can be defined as $\neg(\neg\phi'\rightarrow\neg\phi)$) and then further expanding the language with $\Box$ or $\lozenge$; (2) adding $\coimplies$ or $\triangle$ to $\mathsf{G}$, then both $\Box$ and $\lozenge$ thus acquiring $\KbiG$ (modal bi-G\"{o}del logic) which is further enriched with $\neg$. The reader may see these relations in Fig.~\ref{fig:logics}.
\begin{figure}
\centering
\[\xymatrix{
{\KbiG}^\mathsf{f}&&\KGsquare\\
\KbiG\ar[urr]|{\neg}\ar[u]^{\mathsf{ff}}&\KG\ar[ul]|{\coimplies/\triangle}&\\
\mathsf{biG}\ar[u]|{\Box,\lozenge}&\KG^c\ar[uur]|{\neg}\ar[u]^{\mathsf{ff}}\ar[ul]|{\coimplies/\triangle}&\Gsquare\ar[uu]|{\Box/\lozenge}\\
&\mathsf{G}\ar[ul]|{\coimplies/\triangle}\ar[u]|{\Box,\lozenge}\ar[ur]|{\neg}
}\]
\caption{Logics in the article. $\mathsf{ff}$ stands for ‘permitting fuzzy frames’. Subscripts on arrows denote language expansions. $/$ stands for ‘or’ and comma for ‘and’.}
\label{fig:logics}
\end{figure}
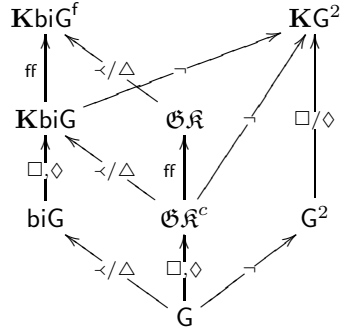

\paragraph{Plan of the paper}
In this paper, we bring together the two sources of inspiration and try to close the gaps outlined above. Namely, we axiomatise the modal bi-G\"{o}del logic over crisp frames $\KbiG$ in the language with $\triangle$ and its paraconsistent expansion $\KGsquare$. We study their semantical properties, establish their decidability, and provide complexity evaluations.

The remainder of the paper is structured as follows. In section~\ref{sec:preliminaries}, we provide the required logical preliminaries for this paper. We define semantics for fuzzy and crisp $\KbiG$ and for crisp $\KGsquare$ and establish some of their useful properties. We also discuss the contribution of $\triangle$ and $\coimplies$ to the expressivity of the $\KbiG$ language in comparison to $\KG$.

In section~\ref{sec:axiomatisation}, we define a~Hilbert-style calculus for crisp $\KbiG$ and establish its weak and strong completeness. Then, we show how to expand our system so as to obtain the complete axiomatisation of crisp $\KGsquare$. We also prove that in the presence of $\neg$, some axioms of $\KbiG$ become redundant in $\KGsquare$.

In section~\ref{sec:semantics}, we investigate the semantical properties of $\KG^c$, $\KbiG$, and $\KGsquare$. In particular, we study transferrable formulas, i.e., formulas classically and G\"{o}del valid on the same classes of frames. We also characterise the class of frames the logics of which allow Glivenko's theorem and its paraconsistent version.

In section~\ref{sec:complexity}, we tackle the decidability and complexity of $\KbiG$ and $\KGsquare$. Using the method of~\cite{CaicedoMetcalfeRodriguezRogger2017}, we prove $\mathsf{PSPACE}$ completeness of satisfiability and validity of $\KbiG$. As a corollary, we obtain $\mathsf{PSPACE}$ completeness of $\KGsquare$.

Finally, in section~\ref{sec:conclusion}, we recapitulate the results obtained in the paper and set the goals for future research.
\section{Preliminaries\label{sec:preliminaries}}
In this section, we provide the semantics of $\KbiG$ and $\KGsquare$ in terms of $[0,1]$-valued Kripke models. We also establish several properties that will help us in the next sections.
\subsection{Semantics of the propositional fragments\label{subsec:propositionalfragments}}
We begin with the semantics of the propositional fragment of $\KbiG$, namely, with $\biG$. The language is generated from the countable set $\Prop$ via the following grammar.
\begin{align*}
\phi&\coloneqq p\in\Prop\mid{\sim}\phi\mid\triangle\phi\mid(\phi\wedge\phi)\mid(\phi\vee\phi)\mid(\phi\rightarrow\phi)\tag{$\Ltriangle$}
\end{align*}
We also introduce two defined constants
\begin{align*}
\mathbf{1}&\coloneqq p\rightarrow p&
\mathbf{0}&\coloneqq{\sim}\mathbf{1}
\end{align*}
%
In our presentation, we choose $\triangle$ over $\coimplies$ as a primitive symbol because the former allows for a shorter and more elegant axiomatisation of the propositional fragment. Furthermore, the use of $\triangle$ simplifies the completeness proof of $\KbiG$. Recall once again the definitions of $\triangle$ and $\coimplies$ via one another.
\begin{align*}
\triangle p&\coloneqq\mathbf{1}\coimplies(\mathbf{1}\coimplies p)&p\coimplies q&\coloneqq p\wedge{\sim}\triangle(p\rightarrow q)
\end{align*}

The semantics of $\Ltriangle$ are given in the following definition. For the sake of simplicity, we also include $\coimplies$ in the definition of bi-G\"{o}del algebras. We remind our readers that we consider $\coimplies$ a defined connective. It will, however, simplify the presentation of the $\Gsquare$ semantics.
\begin{definition}
The bi-G\"{o}del algebra $[0,1]_{\mathsf{G}}=\langle[0,1],0,1,\wedge_\mathsf{G},\vee_\mathsf{G},\rightarrow_{\mathsf{G}},\coimplies,\sim_\mathsf{G},\triangle_\mathsf{G}\rangle$ is defined as follows: for all $a,b\in[0,1]$, the standard operations are given by $a\wedge_\mathsf{G}b\coloneqq\min(a,b)$, $a\vee_\mathsf{G}b\coloneqq\max(a,b)$,
\begin{align*}
a\rightarrow_{G}b&=
\begin{cases}
1,\text{ if }a\leq b\\
b\text{ else}
\end{cases}
&
a\coimplies_{G}b&=
\begin{cases}
0,\text{ if }a\leq b\\
a\text{ else}
\end{cases}
&
{\sim}a&=
\begin{cases}
0,\text{ if }a>0\\
1\text{ else}
\end{cases}
&
\triangle a&=
\begin{cases}
0,\text{ if }a<1\\
1\text{ else}
\end{cases}
\end{align*}
A \emph{$\biG$ valuation} is a homomorphism $e:\Ltriangle\rightarrow[0,1]_\mathsf{G}$ that is defined for the complex formulas as $e(\phi\circ\phi')=e(\phi)\circ_\mathsf{G}e(\phi')$ for every connective $\circ$. We say that $\phi$ is \emph{valid} iff $e(\phi)=1$ under every valuation. Moreover, $\Gamma\subseteq\Ltriangle$ \emph{entails} $\chi\in\Ltriangle$ ($\Gamma\models_{\biG}\chi$) iff for every valuation $e$, it holds that
\[\inf\{e(\phi):\phi\in\Gamma\}\leq e(\chi).\]
\end{definition}
\begin{remark}\label{rem:1preservationentailment}
Note that in contrast to G\"{o}del logic, $\models_{\biG}$ cannot be defined via the preservation of $1$. Indeed, it is easy to check that $e(p\wedge{\sim}\triangle p)<1$ for every $e$, whence an arbitrary formula would have followed from $p\wedge{\sim}\triangle p$. On the other hand, it is clear that $p\wedge{\sim}\triangle p\not\models_{\biG}q$ since if $e(p)=\frac{1}{2}$ and $e(q)=0$, we have that $e(p\wedge{\sim}\triangle p)>e(q)$.
\end{remark}

In order to obtain the paraconsistent expansion of $\biG$, we add $\neg$ to $\Ltriangle$. We dub the resulting language $\Ltrianglesquare$. The semantics of $\Gsquare$ is as follows.
\begin{definition}\label{def:Gsquaresemantics}
A $\Gsquare$ model is a tuple $\langle[0,1],e_1,e_2\rangle$ with $e_1,e_2:\Prop\rightarrow[0,1]$ being extended on the complex formulas as follows.
\begin{longtable}{rclrcl}
$e_1(\neg\phi)$&$=$&$e_2(\phi)$&$e_2(\neg\phi)$&$=$&$e_1(\phi)$\\
$e_1(\phi\wedge\phi')$&$=$&$e_1(\phi)\wedge_\mathsf{G}e_1(\phi')$&$e_2(\phi\wedge\phi')$&$=$&$e_2(\phi)\vee_\mathsf{G}e_2(\phi')$\\
$e_1(\phi\vee\phi')$&$=$&$e_1(\phi)\vee_\mathsf{G}e_1(\phi')$&$e_2(\phi\vee\phi')$&$=$&$e_2(\phi)\wedge_\mathsf{G}e_2(\phi')$\\
$e_1(\phi\rightarrow\phi')$&$=$&$e_1(\phi)\!\rightarrow_\mathsf{G}\!e_1(\phi')$&$e_2(\phi\rightarrow\phi')$&$=$&$e_2(\phi')\coimplies_\mathsf{G}e_2(\phi)$\\
$e_1({\sim}\phi)$&$=$&${\sim_\mathsf{G}}e_1(\phi)$&$e_2({\sim}\phi)$&$=$&$1\coimplies_\mathsf{G}e_2(\phi')$\\
$e_1(\triangle\phi)$&$=$&$\triangle_\mathsf{G}e_1(\phi)$&$e_2(\triangle\phi)$&$=$&${\sim_\mathsf{G}\sim_\mathsf{G}}e_2(\phi')$
\end{longtable}
$\phi\in\Ltrianglesquare$ is \emph{valid} iff for every $\Gsquare$ model, $e_1(\phi)=1$ and $e_2(\phi)=0$. $\Gamma\subseteq\Ltrianglesquare$ \emph{entails} $\chi\in\Ltrianglesquare$ ($\Gamma\models_{\Gsquare}\phi$) iff
\[\inf\{e_1(\phi):\phi\in\Gamma\}\leq e_1(\chi)\text{ and }\sup\{e_2(\phi):\phi\in\Gamma\}\geq e_2(\chi).\]
When there is no risk of confusion, we will write $e(\phi)=(x,y)$ as a shorthand for $e_1(\phi)=x$ and $e_2(\phi)=y$.
\end{definition}
Observe that $e_1$ and $e_2$ in the previous definition can be construed as support of truth and support of falsity of the given formula. Under this condition, we can intuitively say that in order for the entailment to be valid, the conclusion should be \emph{at least as true} and \emph{at most as false} as the premises.
\begin{convention}
To facilitate the presentation, we will introduce the following shorthands. Let $\phi,\phi'\in\Ltrianglesquare$, we set
\begin{align*}
e(\phi)\leq^\pm e(\phi')&\text{ iff }e_1(\phi)\leq e_1(\phi')\text{ and }e_2(\phi)\geq e_2(\phi')\\
e(\phi)<^\pm e(\phi')&\text{ iff }e(\phi)\leq^\pm e(\phi')\text{ and }e(\phi)\neq e(\phi')
\end{align*}
\end{convention}
The next statements are straightforward generalisations of the results in~\cite{Wansing2008}. First, we note that $\Gsquare$ has the $\neg$ negation normal form property.
\begin{proposition}\label{prop:GsquareNNF}
For every formula $\phi\in\Ltrianglesquare$ there is a formula $\NNF(\phi)$ s.t.\ all its $\neg$'s are applied to variables only, and for every $\Gsquare$ model, it holds that
\[e_1(\phi)=e_1(\NNF(\phi))\text{ and }e_2(\phi)=e_2(\NNF(\phi)).\]
\end{proposition}
\begin{proof}
We introduce a shorthand $\phi\leftrightarrow\chi\coloneqq(\phi\rightarrow\chi)\wedge(\chi\rightarrow\phi)$. Observe that $e(\phi\leftrightarrow\chi)=(1,0)$ iff $e(\phi)=e(\chi)$. It is now easy to check that the following formulas are valid.
\begin{align*}
\neg\neg\phi\leftrightarrow\phi&&\neg(\phi\vee\chi)\leftrightarrow(\neg\phi\wedge\neg\chi)&&\neg(\phi\wedge\chi)\leftrightarrow(\neg\phi\vee\neg\chi)\\
\neg\triangle\phi\leftrightarrow{\sim\sim}\neg\phi&&
\neg{\sim}\phi\leftrightarrow{\sim}\triangle\neg\phi&&
\neg(\phi\rightarrow\chi)\leftrightarrow(\neg\chi\wedge{\sim}\triangle(\neg\chi\rightarrow\neg\phi))
\end{align*}
As this shows that $\neg$ can be pushed inside every other connective, the result follows.
\end{proof}
\begin{proposition}[{\cite[Corollary~1]{BilkovaFrittellaKozhemiachenko2021}}]\label{prop:+isenoughGsquare}
$\phi$ is valid iff $e_1(\phi)=1$ for any $\Gsquare$ model.
\end{proposition}

The following statement is an immediate consequence of Proposition~\ref{prop:+isenoughGsquare}.
\begin{proposition}\label{prop:+translationGsquare}
Let $\phi$ be in $\NNF$ and denote with $\phi^+$ the result of the replacement of every negated variable $\neg p$ with a fresh variable $p^*$. Then $\phi$ is $\Gsquare$ valid iff $\phi^+$ is $\biG$ valid.
\end{proposition}

\subsection{Axiomatisation of the propositional fragment}
Let us now define the Hilbert-style calculi for $\biG$ and $\Gsquare$. First, we recall from~\cite{Baaz1996} the $\Ltriangle$ axiomatisation of $\biG$ which we call $\HGtriangle$.
\begin{definition}[$\HGtriangle$ --- Hilbert-style calculus for $\biG$]
The calculus has the following axiom schemas and rules (for any $\phi$, $\chi$,~$\psi$):
\begin{enumerate}
\item $(\phi\rightarrow\chi)\rightarrow((\chi\rightarrow\psi)\rightarrow(\phi\rightarrow\psi))$
\item $\phi\rightarrow(\phi\vee\chi)$; $\chi\rightarrow(\phi\vee\chi)$
\item $(\phi\rightarrow\psi)\rightarrow((\chi\rightarrow\psi)\rightarrow((\phi\vee\chi)\rightarrow\psi))$
\item $(\phi\wedge\chi)\rightarrow\phi$; $(\phi\wedge\chi)\rightarrow\chi$
\item $(\phi\rightarrow\chi)\rightarrow((\phi\rightarrow\psi)\rightarrow(\phi\rightarrow(\chi\wedge\psi)))$
\item $(\phi\rightarrow(\chi\rightarrow\psi))\rightarrow((\phi\wedge\chi)\rightarrow\psi)$; $((\phi\wedge\chi)\rightarrow\psi)\rightarrow(\phi\rightarrow(\chi\rightarrow\psi))$
\item $(\phi\rightarrow\chi)\rightarrow({\sim}\chi\rightarrow{\sim}\phi)$
\item $(\phi\rightarrow\chi)\vee(\chi\rightarrow\phi)$
\item $\triangle\phi\vee{\sim}\triangle\phi$
\item $\triangle(\phi\rightarrow\chi)\rightarrow(\triangle\phi\rightarrow\triangle\chi)$; $\triangle(\phi\vee\chi)\rightarrow(\triangle\phi\vee\triangle\chi)$
\item $\triangle\phi\rightarrow\phi$; $\triangle\phi\rightarrow\triangle\triangle\phi$
\item[MP] $\dfrac{\phi\quad\phi\rightarrow\chi}{\chi}$
\item[$\triangle$nec] $\dfrac{\vdash\phi}{\vdash\triangle\phi}$
\end{enumerate}
\end{definition}
\begin{remark}\label{rem:semilinearity}
Note that instead of $\triangle$ it is possible to treat $\Gtriangle$ as $\biG$ bi-Intuitionistic logic~\cite{Rauszer1974,Rauszer1977}\footnote{The name ‘bi-Intuitionistic’ is actually due to~\cite{Gore2000}.} with \emph{two} linearity axioms: $(p\rightarrow q)\vee(q\rightarrow p)$ and $\mathbf{1}\coimplies((p\coimplies q)\wedge(q\coimplies p))$ (cf., e.g.,~\cite{GrigoliaKiseliovaOdisharia2016}). It is crucial to add both these axioms. In fact, adding only $(p\rightarrow q)\vee(q\rightarrow p)$ results in the axiomatisation of \emph{semi-linear} bi-Heyting algebras (cf.~\cite{BezhanishviliMartinsMoraschini2022} for semi-linear extensions of bi-Intuitionistic logic and~\cite{Beazer1980} for semi-linear bi-Heyting algebras) and, respectively, semi-linear bi-Intuitionistic Kripke frames.
\end{remark}

Let us state several important properties of $\HGtriangle$ that we will utilise in the following parts of the paper.
\begin{proposition}\label{prop:HbiGcompleteness}
$\HGtriangle$ is strongly complete: for any $\Gamma\cup\{\phi\}\subseteq\Ltriangle$, it holds that
\[\Gamma\vdash_{\HGtriangle}\phi\text{ iff }\Gamma\models_{\biG}\phi.\]
\end{proposition}
\begin{remark}\label{rem:trianglededuction}
Note that it is crucial for the soundness of $\HGtriangle$ that $\triangle$nec is applied only to theorems. Otherwise, we would derive $q$ from $p\wedge{\sim}\triangle p$ as follows (but $p\wedge{\sim}\triangle p\not\models_{\biG}q$ as discussed in Remark~\ref{rem:1preservationentailment}).
\begin{enumerate}
\item $p\wedge{\sim}\triangle p$ --- assumption.
\item $p$ --- from 1.
\item ${\sim}\triangle p$ --- from 1.
\item $\triangle p$ --- from 2 by $\triangle$nec.
\item $q$ --- from 3 and 4.
\end{enumerate}
\end{remark}

The calculus for $\Gsquare$ can be easily obtained from $\HGtriangle$: we only need to add De Morgan postulates for the propositional connectives.
\begin{definition}[$\HGsquare$ --- Hilbert-style calculus for $\Gsquare$]
The calculus consists of the following axioms and rules.
\begin{description}
\item[A0:] All instances of $\HGtriangle$ rules and axioms in $\Ltrianglesquare$ language.
\item[$\mathsf{neg}$:] $\neg\neg\phi\leftrightarrow\phi$
\item[$\mathsf{DeM}\wedge$:] $\neg(\phi\wedge\chi)\leftrightarrow(\neg\phi\vee\neg\chi)$
\item[$\mathsf{DeM}\vee$:] $\neg(\phi\vee\chi)\leftrightarrow(\neg\phi\wedge\neg\chi)$
\item[$\mathsf{DeM}\!\rightarrow$:] $\neg(\phi\rightarrow\chi)\leftrightarrow(\neg\chi\wedge{\sim}\triangle(\neg\chi\rightarrow\neg\phi))$
\item[$\mathsf{DeM}\triangle$:] $\neg\triangle\phi\leftrightarrow{\sim\sim}\neg\phi$
\item[$\mathsf{DeM}{\sim}$:] $\neg{\sim}\phi\leftrightarrow{\sim}\triangle\neg\phi$
\end{description}
\end{definition}

The completeness result for $\HGsquare$ formulated with $\coimplies$ instead of $\triangle$ was provided in~\cite{BilkovaFrittellaKozhemiachenkoMajer2022IJAR}. The proof followed the technique from~\cite{Wansing2008} that relied on the existence of $\NNF$'s in $\Gsquare$, Proposition~\ref{prop:+translationGsquare}, and the completeness of the Hilbert-style calculus for $\biG$. But the axioms of $\HGsquare$ in $\Ltrianglesquare$ contain the $\NNF$ transformations as well. Likewise, $\HGtriangle$ is strongly complete. Thus, we can state the strong completeness of $\HGsquare$.
\begin{proposition}\label{prop:HGsquarecompleteness}
$\HGsquare$ is strongly complete: for any $\Gamma\cup\{\phi\}\subseteq\Ltrianglesquare$, it holds that
\[\Gamma\vdash_{\HGsquare}\phi\text{ iff }\Gamma\models_{\Gsquare}\phi.\]
\end{proposition}

We end this section by establishing the following fact.
\begin{proposition}\label{prop:Gsquarecontraposition}
The following rule is admissible in $\Gsquare$:
\begin{align*}
\dfrac{\HGsquare\vdash\phi\rightarrow\chi}{\HGsquare\vdash\neg\chi\rightarrow\neg\phi}
\end{align*}
\end{proposition}
\subsection{Semantics of the modal expansions}
Let us now provide semantics of $\KbiG$ (both fuzzy and crisp) and crisp $\KGsquare$. The language $\bimodalLtrianglesquare$ is defined via the following grammar.
\begin{align*}
\phi&\coloneqq p\in\Prop\mid\neg\phi\mid{\sim}\phi\mid\triangle\phi\mid(\phi\wedge\phi)\mid(\phi\vee\phi)\mid(\phi\rightarrow\phi)\mid\Box\phi\mid\lozenge\phi
\end{align*}
Two constants, $\mathbf{0}$ and $\mathbf{1}$, can be introduced as in section~\ref{subsec:propositionalfragments}. The $\neg$-less fragment of $\bimodalLtrianglesquare$ is denoted with $\bimodalLtriangle$.
\begin{definition}[Frames]\label{def:frames}
\begin{itemize}
\item[]
\item A \emph{fuzzy frame} is a tuple $\mathfrak{F}=\langle W,R\rangle$ with $W\neq\varnothing$ and $R:W\times W\rightarrow[0,1]$.
\item A \emph{crisp frame} is a tuple $\mathfrak{F}=\langle W,R\rangle$ with $W\neq\varnothing$ and $R\subseteq W\times W$.
\end{itemize}
\end{definition}
\begin{definition}[$\KbiG$ models]\label{def:KbiGsemantics}
A \emph{$\KbiG$ model} is a tuple $\mathfrak{M}=\langle W,R,e\rangle$ with $\langle W,R\rangle$ being a~(crisp or fuzzy) frame, and $e:\mathsf{Var}\times W\rightarrow[0,1]$. $e$ (a valuation) is extended on complex $\bimodalLtriangle$ formulas as follows:
\begin{align*}
e(\phi\circ\phi',w)&=e(\phi,w)\circ_\mathsf{G}e(\phi',w).\tag{$\circ\in\{{\sim},\triangle,\wedge,\vee,\rightarrow\}$}
\end{align*}
The interpretation of modal formulas on \emph{fuzzy} frames is as follows:
\begin{align*}
e(\Box\phi,w)&=\inf\limits_{w'\in W}\{wRw'\rightarrow_\mathsf{G}e(\phi,w')\},&e(\lozenge\phi,w)&=\sup\limits_{w'\in W}\{wRw'\wedge_\mathsf{G}e(\phi,w')\}.
\end{align*}
On \emph{crisp} frames, the interpretation is simpler (here, $\inf(\varnothing)\!=\!1$ and $\sup(\varnothing)\!=\!0$):
\begin{align*}
e(\Box\phi,w)&=\inf\{e(\phi,w'):wRw'\},&e(\lozenge\phi,w)&=\sup\{e(\phi,w'):wRw'\}.
\end{align*}
We say that $\phi\in\bimodalLtriangle$ is \emph{$\KbiG$ valid on frame $\mathfrak{F}$} (denote, $\mathfrak{F}\models_{\KbiG}\phi$) iff for any $w\in\mathfrak{F}$, it holds that $e(\phi,w)=1$ for any model $\mathfrak{M}$ on $\mathfrak{F}$. $\Gamma$ \emph{entails} $\chi$ (on $\mathfrak{F}$), denoted $\Gamma\models_{\KbiG}\phi$ ($\Gamma\models^{\mathfrak{F}}_{\KbiG}\chi$), iff for every model $\mathfrak{M}$ (on $\mathfrak{F}$) and every $w\in\mathfrak{M}$, it holds that
\[\inf\{e(\phi,w):\phi\in\Gamma\}\leq e(\chi,w).\]
\end{definition}

In what follows, we use $\KbiG$ to stand for the set of all $\bimodalLtriangle$ formulas valid on all \emph{crisp} frames and $\KbiG^\mathsf{f}$ to stand for the set of all $\bimodalLtriangle$ formulas valid on all \emph{fuzzy} frames.

In~\cite{BilkovaFrittellaKozhemiachenko2022IJCAR}, we argued that one can think of \emph{crisp} $R$ as availability of trusted sources represented by states in the model. \emph{Fuzzy} accessibility relation can be interpreted as the degree of trust an agent has in a source.

Since sources can refer to one another and can consider one another more or less reliable, we can understand modalities as follows. $\lozenge\phi$ is the search for evidence that supports $\phi$ from trusted sources: $e(\lozenge\phi,t)>0$ iff there is a source $t'$ to which $t$ has positive degree of trust and that has at least some certainty in $\phi$. If, however $t$ trusts nobody (i.e., $tRu=0$ for all $u$), then $e(\lozenge\phi,t)=0$. Similarly, $\Box\chi$ represents the search of evidence given by trusted sources that does not support $\chi$: $e(\Box\chi,t)<1$ iff there is a~source $t'$ that gives to $\chi$ less certainty than $t$ gives trust to $t'$. This means that if $t$ trusts no sources, or if all sources have at least as high confidence in $\chi$ as $t$ has in them, then $t$ fails to find a trustworthy enough counterexample.
\begin{definition}[$\KGsquare$ models]\label{def:KG2semantics}
A \emph{$\KGsquare$ model} is a tuple $\mathfrak{M}=\langle W,R,e_1,e_2\rangle$ with $\langle W,R\rangle$ being a \emph{crisp} frame, and $e_1,e_2:\mathsf{Var}\times W\rightarrow[0,1]$. The valuations which we interpret as support of truth and support of falsity, respectively, are extended on complex formulas as expected.

Namely, the propositional connectives are defined state-wise according to definition~\ref{def:Gsquaresemantics}. The modalities are defined as follows.
\begin{align*}
e_1(\Box\phi,w)&=\inf\{e_1(\phi,w'):wRw'\}&e_2(\Box\phi,w)&=\sup\{e_2(\phi,w'):wRw'\}\\
e_1(\lozenge\phi,w)&=\sup\{e_1(\phi,w'):wRw'\}&e_2(\lozenge\phi,w)&=\inf\{e_2(\phi,w'):wRw'\}
\end{align*}
We say that $\phi\in\bimodalLtrianglesquare$ is \emph{$\KGsquare$ valid on frame $\mathfrak{F}$} ($\mathfrak{F}\models_{\KGsquare}\phi$) iff for any $w\in\mathfrak{F}$, it holds that $e_1(\phi,w)=1$ and $e_2(\phi,w)=0$ for any model $\mathfrak{M}$ on $\mathfrak{F}$. $\Gamma$ \emph{entails} $\chi$ (on $\mathfrak{F}$), denoted $\Gamma\models_{\KGsquare}\phi$ ($\Gamma\models^{\mathfrak{F}}_{\KGsquare}\chi$), iff for every model $\mathfrak{M}$ (on $\mathfrak{F}$) and every $w\in\mathfrak{M}$, it holds that
\[\inf\{e_1(\phi,w):\phi\in\Gamma\}\leq e_1(\chi,w)\text{ and }\sup\{e_2(\phi,w):\phi\in\Gamma\}\geq e_2(\chi,w).\]
\end{definition}
\begin{remark}
Note that $\KGsquare$ entailment is paraconsistent in the following sense: if $\phi$ is not valid, then there is some $\chi$ s.t.\
$\phi,\neg\phi\not\models_{\KGsquare}\chi$ (i.e., the entailment is not explosive w.r.t.\ $\neg$). This accounts for the possibility of the sources giving contradictory information.

Furthermore, in contrast to $\mathbf{K}$, the agent can believe in contradictions \emph{in a non-trivial manner} as we have that $\Box(p\wedge\neg p)\not\models_{\KGsquare}\Box q$ and $\lozenge(p\wedge\neg p)\not\models_{\KGsquare}\lozenge q$.
\end{remark}
\begin{convention}
For each frame $\mathfrak{F}$ and each $w\in\mathfrak{F}$, we denote
\begin{align*}
R(w)&=\{w':wRw'=1\}\tag{for fuzzy frames}\\
R^+(w)&=\{w':wRw'>0\}\tag{for fuzzy frames}\\
R(w)&=\{w':wRw'\}\tag{for crisp frames}
\end{align*}
\end{convention}

Observe that $e(\Box\phi,w)=e(\neg\lozenge\neg\phi,w)$. Thus, we can treat $\lozenge$ as a defined connective. Furthermore, this means that $\KGsquare$ has the $\neg$ $\NNF$ property as well, and that the following statement holds.
\begin{proposition}[{\cite[Proposition~1]{BilkovaFrittellaKozhemiachenko2022IJCAR}}]\label{prop:+isenough}
$\mathfrak{F}\models_{\KGsquare}\!\phi$ iff for any model $\mathfrak{M}$ on $\mathfrak{F}$ and any $w\!\in\!\mathfrak{F}$, $e_1(\phi,w)\!=\!1$.
\end{proposition}
In fact, we can reduce $\KGsquare$ validity to $\KbiG$ validity in the same manner as we did for their propositional fragments.
\begin{proposition}\label{prop:+translation}
Let $\mathfrak{F}$ be a crisp frame and $\phi$ be in $\NNF$. Then $\mathfrak{F}\models_{\KGsquare}\phi$ iff $\mathfrak{F}\models_{\KbiG}\phi^+$ for any $\phi\in\bimodalLtrianglesquare$.
\end{proposition}
\begin{proof}
By Proposition~\ref{prop:+isenough}, we have that $\mathfrak{F}\models_{\KGsquare}\phi$ iff in every $\KGsquare$ model $\mathfrak{M}$ on $\mathfrak{F}$ and every $w\in\mathfrak{M}$ it holds that $e_1(\phi,w)=1$. It remains to construct a $\KbiG$ model $\mathfrak{M}^+$ on the same frame where $e_1(\phi,u)=e^+(\phi^+,u)$ for every $\phi$ and $u$.

For any $u\in W$, define the valuation $e^+$ as follows:
\begin{align*}
e^+(p,u)&=e_1(p,u)\\e^+(p^*,u)&=e_2(p,u)
\end{align*}

It now suffices to show that $e^+(\phi^+,u)=e_1(\phi,u)$ for any $\phi$ and $w$. We proceed by induction on $\phi$. The basis cases of literals are straightforward as well as those of the propositional connectives. Thus, we consider the case of $\phi=\Box\phi'$.
\begin{align*}
e_1(\Box\phi',u)&=\inf\{e_1(\phi',u'):uRu'\}\\
&=\inf\{e^+(\phi'^+,u'):uRu'\}\tag{by IH}\\
&=e^+(\Box\phi'^+,u)
\end{align*}

The case of $\phi=\lozenge\phi'$ can be considered in the same manner.
\end{proof}

We end the section by recalling the conservativity results.
\begin{proposition}[{\cite[Proposition~2]{BilkovaFrittellaKozhemiachenko2022IJCAR}}]\label{prop:conservativity}
\begin{enumerate}
\item[]
\item Let $\phi$ be a formula over $\{\mathbf{0},\wedge,\vee,\rightarrow,\Box,\lozenge\}$. Then, $\mathfrak{F}\models_{\KG}\phi$ iff $\mathfrak{F}\models_{\KbiG^\mathsf{f}}\phi$ and $\mathfrak{F}\models_{\KG^c}\phi$ iff $\mathfrak{F}\models_{\KbiG}\phi$, for any~$\mathfrak{F}$.
\item Let $\phi\in\bimodalLtriangle$. Then, $\mathfrak{F}\models_{\KbiG}\phi$ iff $\mathfrak{F}\models_{\KGsquare}\phi$, for any crisp~$\mathfrak{F}$.
\end{enumerate}
\end{proposition}
\subsection{Expressivity of $\triangle$}
We have added $\triangle$ to the language of G\"{o}del modal logic. It is thus instructive to investigate whether it gives us the expressive capacity one does not have without it.

First of all, it is easy to see that $\triangle$ allows us to express the statements of \emph{comparative belief}. For example\footnote{More examples and a more detailed discussion of such statements can be found in~\cite{BilkovaFrittellaKozhemiachenko2022IJCAR}.}, consider the following statement
\begin{quote}
\textsf{weather}: \emph{Paula considers a rain happening today strictly more likely than a hailstorm}.
\end{quote}
Thus, to formalise this statement, one needs a formula that is true iff the value of $\Box r$ (Paula believes it is going to rain today) is strictly greater than that of $\Box s$ (Paula believes that a hailstorm is going to happen today). Paula also does not state that she \emph{believes completely} in the rain, nor does she exclude the possibility of a hailstorm. Hence, $\Box r\wedge\Box{\sim}s$ does not suit the purpose. In fact, there is no G\"{o}del formula $\phi(p,q)$ s.t.\
\[e(\phi)=1\text{ iff }e(p)>e(q)\]
On the other hand, it is easy to see that
\[e({\sim}\triangle(\Box r\rightarrow\Box s),w)=1\text{ iff }e(\Box r,w)>e(\Box s,w)\]
and thus is a suitable formalisation of \textsf{weather}.

It is also possible to formalise comparative statements in the $\Gsquare$ (and hence, $\KGsquare$) setting. Notice, first, that when we consider support of truth and support of falsity \emph{independently}, it is no longer the case that every two beliefs are comparable. This, actually, aligns with our intuition: indeed, if the contents of two statements have no connection to each other, an agent might not be ready to choose one that they find more believable.

Formally, we can represent this as follows. Define
\begin{align*}
\triangle^\neg\phi\coloneqq\triangle\phi\wedge\neg{\sim}\triangle\phi
\end{align*}
One can see that
\begin{align*}
e(\triangle^\neg\phi,w)&=
\begin{cases}
(1,0)&\text{if }e(\phi,w)=(1,0)\\
(0,1)&\text{otherwise}
\end{cases}
\end{align*}
and that $\triangle^\neg(p\rightarrow q)\vee\triangle^\neg(q\rightarrow p)$ is not $\Gsquare$ valid while $\triangle(p\rightarrow q)\vee\triangle(q\rightarrow p)$ \emph{is $\biG$ valid}. Now, to formalise \textsf{weather} in a $\KGsquare$ setting, we use the following formula.
\[\psi\coloneqq\triangle^\neg(\Box s\rightarrow\Box r)\wedge{\sim}\triangle^\neg(\Box r\rightarrow\Box s)\]
One can check that, indeed
\begin{align*}
e(\psi,w)&=
\begin{cases}
(1,0)&\text{ iff }e(\Box s,w)<^\pm e(\Box r,w)\\
(0,1)&\text{ otherwise}
\end{cases}
\end{align*}

As we have just seen, the addition of $\triangle$ allows us to formalise the statements we were not able to treat without it. On a more formal side, however, $\triangle$ makes both $\Box$ and $\lozenge$ fragments\footnote{Note that $\Box$ and $\lozenge$ are not interdefinable in $\KG^c$~\cite[Corollary~6.2]{RodriguezVidal2021}, nor in $\KbiG$~\cite[Corollary~2]{BilkovaFrittellaKozhemiachenko2022IJCAR}.} of $\KbiG$ more expressive. Namely, $\lozenge$ fragment of $\KG$ has finite model property while crisp and fuzzy $\Box$ fragments coincide~\cite{CaicedoRodriguez2010}. We show that neither of these is the case in $\KbiG$.
\begin{proposition}\label{prop:triangleexpressive}
\begin{enumerate}
\item[]
\item $\mathfrak{F}\models_{\KbiG}\triangle\Box p\rightarrow\Box\triangle p$ iff $\mathfrak{F}$ is crisp.
\item There are only infinite countermodels of $\triangle\lozenge p\rightarrow\lozenge\triangle p$.
\end{enumerate}
\end{proposition}
\begin{proof}
We begin with $1$. Assume that $\mathfrak{F}$ is crisp, and let $e$ be a valuation thereon s.t.\ $e(\triangle\Box p,w)=1$. Then, $e(\Box p,w)=1$. But $\mathfrak{F}$ is crisp, whence, $e(p,w')=1$ and thus, $e(\triangle p,w')=1$ for every accessible $w'$. Thus, $e(\Box\triangle p,w)=1$, as required. For the converse, assume that $\mathfrak{F}$ is \emph{fuzzy} and that w.l.o.g.\ $wRw'=\frac{1}{2}$. We refute $\triangle\Box p\rightarrow\Box\triangle p$ at $w$ as follows. Set $e(p,w')=\frac{2}{3}$ and $e(p,w'')=1$ in all other states. It is clear that $e(\triangle\Box p,w)=1$ but $e(\Box\triangle p,w)=0$ for we have $e(\triangle p,w')=0$.

For $2$, we proceed as follows. Let $\mathfrak{M}$ be a finite model and let $e(\triangle\lozenge p,w)=1$. Then, there is $w'\in R(w)$ s.t.\ $e(p,w')=1$, whence $e(\lozenge\triangle p,w)=1$. For the converse, assume that $e(\triangle\lozenge p,w)=1$ and $e(\triangle\lozenge p,w)<1$. We define an infinite fuzzy\footnote{Recall from~\cite{CaicedoRodriguez2010} that the crisp $\lozenge$ fragment of $\KG$ lacks FMP.} countermodel as follows.
\begin{itemize}
\item $W=\{w\}\cup\{w_i:i\in\mathbb{N}\text{ and }i\geq1\}$.
\item $wRw_i=\frac{i}{i+1}$; $uRu'=0$ for every $u\neq w$ and $u\neq w_i$.
\item $e(p,w_i)=\frac{i}{i+1}$.
\end{itemize}
It is clear that this model is infinite and that $e(\triangle\lozenge p\rightarrow\lozenge\triangle p,w)=0$.
\end{proof}
\begin{remark}
Note that it is also easy to show that $\lozenge\triangle p\rightarrow\triangle\lozenge p$ defines crisp frames but, of course, one can define crisp frames without $\triangle$: ${\sim\sim}\lozenge p\rightarrow\lozenge{\sim\sim}p$~\cite{CaicedoRodriguez2010}.
\end{remark}
\section{Axiomatisation of $\KbiG$ and $\KGsquare$\label{sec:axiomatisation}}
We are now finally ready to formulate Hilbert-style calculi for crisp $\KbiG$ and $\KGsquare$ and prove their completeness. Our completeness proof follows the approach of~\cite{CaicedoRodriguez2015} and~\cite{RodriguezVidal2021}. Note, however, that we cannot completely copy the original proof from~\cite{RodriguezVidal2021} because it employs that the entailment in G\"{o}del logic can be \emph{equivalently} defined either as preservation of the order on $[0,1]$ or as preservation of $1$ as the designated value. This, however, is not true of modal expansions of $\biG$ as we have seen in Remarks~\ref{rem:1preservationentailment} and~\ref{rem:trianglededuction}

We begin with the calculus for $\KbiG$ which we dub $\HKbiG$.
\begin{definition}[$\HKbiG$ --- Hilbert-style calculus for $\KbiG$]\label{def:HKbiG}
The calculus has the following axiom schemas and rules.
\begin{description}
\item[$\biG$:] All substitution instances of $\HGtriangle$ theorems and rules.
\item[$\mathbf{0}$:] ${\sim}\lozenge\mathbf{0}$
\item[K:] $\Box(\phi\rightarrow\chi)\rightarrow(\Box\phi\rightarrow\Box\chi)$; $\lozenge(\phi\vee\chi)\rightarrow(\lozenge\phi\vee\lozenge\chi)$
\item[FS:] $\lozenge(\phi\rightarrow\chi)\rightarrow(\Box\phi\rightarrow\lozenge\chi)$; $(\lozenge\phi\rightarrow\Box\chi)\rightarrow\Box(\phi\rightarrow\chi)$
\item[${\sim}\triangle\lozenge$:] ${\sim}\triangle(\lozenge\phi\rightarrow\lozenge\chi)\rightarrow\lozenge{\sim}\triangle(\phi\rightarrow\chi)$
\item[Cr:] $\Box(\phi\vee\chi)\rightarrow(\Box\phi\vee\lozenge\chi)$; $\triangle\Box\phi\rightarrow\Box\triangle\phi$
\item[nec:] $\dfrac{\vdash\phi}{\vdash\Box\phi}$; $\dfrac{\vdash\phi\rightarrow\chi}{\vdash\lozenge\phi\rightarrow\lozenge\chi}$
\end{description}
\end{definition}

As one sees from the definition above, we have added two modal axioms to the Hilbert-style calculus $\HKGc$ ($\mathcal{GK}^c$, in the notation of~\cite{RodriguezVidal2021}) that axiomatises $\KG^c$. ${\sim}\triangle\lozenge$ says that if the supremum of $\phi$ is strictly greater than supremum of $\chi$, then there must be a~state where the value of $\phi$ is greater than that of $\chi$. The second axiom is the definition of crisp frames without $\lozenge$ but with $\triangle$.

In what follows, we denote the set of $\HKbiG$ theorems (i.e., formulas provable without assumptions) with $\Th(\HKbiG)$. Observe that $\Th(\HKGc)\!\subseteq\!\Th(\HKbiG)$. In particular, $\mathsf{P}\!\coloneqq\!\Box(\phi\!\rightarrow\!\chi)\!\rightarrow\!(\lozenge\phi\!\rightarrow\!\lozenge\chi)$ is provable and $$\MBox:\dfrac{\Gamma\vdash\phi}{\Box\Gamma\vdash\Box\phi}$$ is admissible. Using this, we obtain the following statement.
\begin{proposition}\label{prop:BarcanCrispness}~
\begin{enumerate}
\item The Barcan's formula $\Box\triangle\phi\rightarrow\triangle\Box\phi$ is provable in $\HKbiG$ \emph{without} using $\mathbf{Cr}$.
\item The $\lozenge\triangle$ definition of crispness $\lozenge\triangle\phi\rightarrow\triangle\lozenge\phi$ is provable in $\HKbiG$.
\end{enumerate}
\end{proposition}
\begin{proof}
We begin with 1. First, observe that the following rule is admissible in $\HKG$.
\[\dfrac{\HKG\vdash{\sim}\phi\vee\phi}{\HKG\vdash{\sim}\Box\phi\vee\Box\phi}\]
Furthermore, the following rule is admissible in $\HGtriangle$:
\[\dfrac{\HGtriangle\vdash{\sim}\phi\vee\phi\quad\HGtriangle\vdash\phi\rightarrow\chi}{\HGtriangle\vdash\phi\rightarrow\triangle\chi}\]

Thus, we can prove the Barcan's formula as follows.
\begin{enumerate}
\item ${\sim}\triangle\phi\vee\triangle\phi$ --- a theorem of $\HGtriangle$
\item ${\sim}\Box\triangle\phi\vee\Box\triangle\phi$ --- from $1$
\item $\triangle\phi\rightarrow\phi$ --- a theorem of $\HGtriangle$
\item $\Box\triangle\phi\rightarrow\Box\phi$ --- from 3, $\mathbf{nec}$, and $\mathbf{K}$
\item $\Box\triangle\phi\rightarrow\triangle\Box\phi$ --- from 2 and 4
\end{enumerate}

To prove the $\lozenge\triangle$ definition of the crispness, we proceed as follows.
\begin{enumerate}
\item ${\sim}\triangle\phi\vee\triangle\phi$ --- a theorem of $\HGtriangle$
\item $\Box({\sim}\triangle\phi\vee\triangle\phi)$ --- from 1 using $\mathbf{nec}$
\item $\Box{\sim}\triangle\phi\vee\lozenge\triangle\phi$ --- from 2 using $\mathbf{Cr}$
\item ${\sim}\lozenge\triangle\phi\vee\lozenge\triangle\phi$ --- from 3 since $\HKG\vdash\Box{\sim}\chi\rightarrow{\sim}\lozenge\phi$
\item $\triangle\phi\rightarrow\phi$ --- a theorem of $\HGtriangle$
\item $\lozenge\triangle\phi\rightarrow\lozenge\phi$ --- from 5 using $\mathbf{K}$
\item $\lozenge\triangle\phi\rightarrow\triangle\lozenge\phi$ --- from 4 and 6
\end{enumerate}
\end{proof}

Furthermore, just as in the case of $\KG$, the modal rules of $\HKbiG$ are restricted to theorems. Thus, we can reduce the proofs in $\HKbiG$ to the $\HGtriangle$ derivations from $\Th(\HKbiG)$.
\begin{proposition}\label{prop:propositionalreduction}
For any $\Gamma\cup\{\phi,\chi\}\subseteq\bimodalLtriangle$, it holds that
\begin{align*}
\Gamma\vdash_{\HKbiG}\phi&\text{ iff }\Gamma,\Th(\HKbiG)\vdash_{\HGtriangle}\phi
\end{align*}
\end{proposition}
We are now ready to prove the completeness theorem. Our proof is a modification of the completeness theorem for crisp G\"{o}del modal logic in~\cite{RodriguezVidal2021}.
\begin{convention}
For any $\phi\in\bimodalLtriangle$, we denote with $\Sfconst(\phi)$ the set containing all its subformulas and the constants $\mathbf{1}$ and $\mathbf{0}$.
\end{convention}

For every $\tau\in\bimodalLtriangle$ s.t.\ $\HKbiG\not\vdash\tau$, we are building a canonical model $\mathfrak{M}^\tau$ that refutes it.
\begin{definition}[Canonical model for $\tau$]
We define $\mathfrak{M}^\tau=\langle W^\tau,R^\tau,e^\tau\rangle$ as follows.
\begin{itemize}
\item $W^\tau$ is the set of all $\Gtriangle$ homomorphisms $u:\bimodalLtriangle\rightarrow[0,1]_{\biG}$ s.t.\ all theorems of $\HKbiG$ are evaluated at $1$.
\item $uR^\tau u'$ iff $u(\Box\psi)\leq u'(\psi)$ and $u'(\psi)\leq u(\lozenge\psi)$ for all $\psi\in\Sfconst(\tau)$.
\item $e^\tau(p,u)=u(p)$.
\end{itemize}
\end{definition}

Following~\cite{RodriguezVidal2021}, we introduce the following notation.
\begin{convention}
Let $u\in W^\tau$, $\alpha\in[0,1]$, $\heartsuit\in\{\Box,\lozenge\}$, and $\curlyvee\in\{<,\leq,>,\geq,=\}$. We set
\begin{align*}
\heartsuit^{\curlyvee\alpha}_u&\coloneqq\{\chi\in\Sfconst(\tau):u(\heartsuit\chi)\curlyvee\alpha\}&{}^*\Box^{=1}_u&\coloneqq\{\psi:u(\Box\psi)=1\}
\end{align*}
Observe that $\heartsuit^{\curlyvee\alpha}_u$ is always finite, whence $\bigwedge\heartsuit^{\curlyvee\alpha}_u,\bigvee\heartsuit^{\curlyvee\alpha}_u\in\bimodalLtriangle$. Furthermore, if $\heartsuit^{\curlyvee\alpha}_u=\varnothing$, we set $\bigwedge\varnothing=\mathbf{1}$ and $\bigvee\varnothing=\mathbf{0}$.
\end{convention}

The following two statements are the analogues of Lemma~4.1 and Remark~4.2 from~\cite{RodriguezVidal2021} and can be proven in exactly the same manner.
\begin{proposition}\label{prop:boxwrapping}
Let $\phi\in\Box^{=\alpha}_u$ with $\alpha<1$ and set \[\delta\coloneqq\left(\bigwedge\Box^{>\alpha}_u\rightarrow\phi\right)\rightarrow\phi\] Then $u(\Box\delta)>\alpha$.
\end{proposition}
\begin{proposition}\label{prop:boxwrappingremark}
For any $\biG$ homomorphism $e$ s.t.\ $e(\delta)=1$ and $e(\phi)<1$, it holds that $e(\phi)<e(\chi)$ for any $\chi\in\Box^{>\alpha}_u$.
\end{proposition}

We are now ready to prove the analogue of~\cite[Proposition~4.3]{RodriguezVidal2021}. Note, however, that we cannot exactly follow the original proof step by step as it uses the fact that the propositional entailment in G\"{o}del logic can be \emph{equivalently} defined either via preservation of the order on $[0,1]$ and via preservation of~$1$. Namely, the original proof is built on failing an instance of G\"{o}del entailment in such a way that the premises are evaluated at $1$. This, of course, is not the case in $\Gtriangle$ for arbitrary formulas as we have noted above. Thus, we need to modify the formulas used in the proof. Namely, instead of $\delta$, we need to use $\triangle\delta$.
\begin{proposition}\label{prop:homomorphismexistencebox}
For any $\alpha\!<\!1$ and $\phi\!\in\!\Box^{=\alpha}_u$, there exists a~propositional homomorphism $h\!:\!\bimodalLtriangle\!\rightarrow\![0,1]_{\biG}$, s.t.:
\begin{description}
\item[C1:] $h(\chi)=1$ for any $\chi\in\Th(\HKbiG)$;
\item[C2:] $h(\psi)=1$ for every $\psi\in{}^*\Box^{=1}_u$;
\item[C3:] $h(\rho)<1$ for every $\rho\in\lozenge^{<1}_u$;
\item[C4:] $h(\phi)<h(\sigma)$ for every $\sigma\in\Box^{>\alpha}_u$.
\end{description}
\end{proposition}
\begin{proof}
Recall that for any $\phi_1$, $\phi_2$, and $\phi_3$, it holds that
\[\HGtriangle\vdash(((\phi_1\rightarrow\phi_2)\rightarrow\phi_2)\wedge(\phi_2\rightarrow\phi_3))\vee(((\phi_1\rightarrow\phi_2)\rightarrow\phi_2)\rightarrow(\phi_3\rightarrow\phi_2))\]
We replace $\phi_1$ with $\bigwedge\Box^{>\alpha}_u$, $\phi_2$ with $\phi$, and use $\delta$ from Proposition~\ref{prop:boxwrapping} which gives us that
\[\HKbiG\vdash\left(\delta\wedge\left(\phi\rightarrow\bigvee\lozenge^{<1}_u\right)\right)\vee\left(\delta\rightarrow\left(\bigvee\lozenge^{<1}_u\rightarrow\phi\right)\right)\]

Since $$\dfrac{\HGtriangle\vdash(\varphi\wedge\varphi')\vee\eta}{\HGtriangle\vdash(\triangle\varphi\wedge\varphi')\vee\eta}$$ is admissible in $\HGtriangle$, we have 
\[\HKbiG\vdash\left(\triangle\delta\wedge\left(\phi\rightarrow\bigvee\lozenge^{<1}_u\right)\right)\vee\left(\delta\rightarrow\left(\bigvee\lozenge^{<1}_u\rightarrow\phi\right)\right)\]

Now, we use the commutativity of $\vee$, apply $\mathbf{nec}$, and then $\mathbf{Cr}$ to obtain that
\[\HKbiG\vdash\lozenge\left(\triangle\delta\wedge\left(\phi\rightarrow\bigvee\lozenge^{<1}_u\right)\right)\vee\Box\left(\delta\rightarrow\left(\bigvee\lozenge^{<1}_u\rightarrow\phi\right)\right)\]
Since $u\in W^\tau$, we have that one of the following holds:
\begin{enumerate}
\item[(A)] $u\left(\lozenge\left(\triangle\delta\wedge\left(\phi\rightarrow\bigvee\lozenge^{<1}_u\right)\right)\right)=1$ or
\item[(B)] $u\left(\Box\left(\delta\rightarrow\left(\bigvee\lozenge^{<1}_u\rightarrow\phi\right)\right)\right)=1$.
\end{enumerate}
We prove the statement in both cases.

Assume that (A) holds. We show that
\begin{align}
\Th(\HKbiG),\triangle({}^*\Box^{=1}_u),\triangle\delta\not\models_{\biG}\left(\phi\rightarrow\bigvee\lozenge^{<1}_u\right)\rightarrow\bigvee\lozenge^{<1}_u\label{equ:3.5.A}
\end{align}

We reason for the contradiction. Note that $\HGtriangle$ is strongly complete w.r.t.\ $\biG$, and that $\HGtriangle\vdash\triangle\delta\rightarrow\triangle\triangle\delta$. Thus, applying Proposition~\ref{prop:propositionalreduction}, we obtain
\[\triangle({}^*\Box^{=1}_u)\vdash_{\HKbiG}\left(\triangle\delta\wedge\left(\phi\rightarrow\bigvee\lozenge^{<1}_u\right)\right)\rightarrow\bigvee\lozenge^{<1}_u\]
We apply $\MBox$, $\mathsf{P}$, and $\mathbf{K}$ and get
\[\Box\triangle({}^*\Box^{=1}_u)\vdash_{\HKbiG}\lozenge\left(\triangle\delta\wedge\left(\phi\rightarrow\bigvee\lozenge^{<1}_u\right)\right)\rightarrow\bigvee\lozenge\lozenge^{<1}_u\]
Now, since $\triangle\Box\phi\rightarrow\Box\triangle\phi$ is an axiom scheme $\mathbf{Cr}$, we have that
\[\triangle\Box({}^*\Box^{=1}_u)\vdash_{\HKbiG}\lozenge\left(\triangle\delta\wedge\left(\phi\rightarrow\bigvee\lozenge^{<1}_u\right)\right)\rightarrow\bigvee\lozenge\lozenge^{<1}_u\]
We apply Proposition~\ref{prop:propositionalreduction} again which gives us that
\[\Th(\HKbiG),\triangle\Box({}^*\Box^{=1}_u)\models_{\biG}\lozenge\left(\triangle\delta\wedge\left(\phi\rightarrow\bigvee\lozenge^{<1}_u\right)\right)\rightarrow\bigvee\lozenge\lozenge^{<1}_u\]
However, we can show that $u$ refutes this entailment. Indeed, observe that $u(\Th(\HKbiG))=\{1\}$ since $u\in W^\tau$. Moreover, since $u(\Box{}^*\Box^{=1}_u)=\{1\}$ by definition, we have that $u(\triangle\Box{}^*\Box^{=1}_u)=1$ as well. Finally, (A) gives us that $u\left(\lozenge\left(\triangle\delta\wedge\left(\phi\rightarrow\bigvee\lozenge^{<1}_u\right)\right)\right)=1$ but $u\left(\bigvee\lozenge\lozenge^{<1}_u\right)<1$ by definition.

Thus, since the premises of \eqref{equ:3.5.A} are either theorems of $\HKbiG$ or formulas whose main connective is $\triangle$, there is a homomorphism $h$ that sends the premises of \eqref{equ:3.5.A} to $1$ and the conclusion to a lesser value. We show that $h$ satisfies the conditions of the statement. Indeed, $\mathbf{C1}$ is obtained immediately since $\Th(\HKbiG)$ is closed under $\triangle$. To see that $\mathbf{C2}$ holds, we note that $h(\triangle({}^*\Box^{=1}_u))=\{1\}$, whence $h({}^*\Box^{=1}_u)=\{1\}$ too.

Since $h$ refutes the conclusion of~\eqref{equ:3.5.A}, we have that $h\left(\bigvee\lozenge^{<1}_u\right)<h\left(\phi\rightarrow\bigvee\lozenge^{<1}_u\right)$. Hence, $h$ satisfies~$\mathbf{C3}$. Finally, $h(\triangle\delta)=1$ entails that $h(\delta)=1$. But one can see that $h(\phi)\leq h\left(\bigvee\lozenge^{<1}_u\right)$, whence $\mathbf{C4}$ also holds w.r.t.\ $h$ by Proposition~\ref{prop:boxwrappingremark}.

We consider (B). We assume that
\begin{align}
\Th(\HKbiG),\triangle({}^*\Box^{=1}_u),\delta,\delta\rightarrow\left(\bigvee\lozenge^{<1}_u\rightarrow\phi\right)\models_{\biG}\phi
\label{equ:3.5.B}
\end{align}
and reason for contradiction. For this, we apply Proposition~\ref{prop:propositionalreduction}, strong completeness of $\HGtriangle$ w.r.t.\ $\biG$, and $\MBox$ to obtain
\begin{align*}
\Th(\HKbiG),\Box\triangle({}^*\Box^{=1}_u),\Box\delta,\Box\left(\delta\rightarrow\left(\bigvee\lozenge^{<1}_u\rightarrow\phi\right)\right)\models_{\biG}\Box\phi
\end{align*}
Now, we apply $\mathbf{Cr}$ which gives us
\begin{align*}
\Th(\HKbiG),\triangle\Box({}^*\Box^{=1}_u),\Box\delta,\Box\left(\delta\rightarrow\left(\bigvee\lozenge^{<1}_u\rightarrow\phi\right)\right)\models_{\biG}\Box\phi
\end{align*}
Again, we can refute this entailment with $u$. Since $u\in W^\tau$, $u(\Th(\HKbiG))=\{1\}$. Furthermore, $u(\triangle\Box({}^*\Box^{=1}_u))=1$ since $u(\Box({}^*\Box^{=1}_u))=1$ by definition of $u$, and $u(\Box\delta)=1$ by Proposition~\ref{prop:boxwrapping}, and $u\left(\Box\left(\delta\rightarrow\left(\bigvee\lozenge^{<1}_u\rightarrow\phi\right)\right)\right)=1$ by assumption (B). On the contrary, $u(\Box\phi)<\alpha$.

Thus, there exists a homomorphism $h$ that evaluates the premises of \eqref{equ:3.5.B} at $1$\footnote{Note that theorems are closed under $\triangle$, that $h\left(\delta\rightarrow\left(\bigvee\lozenge^{<1}_u\rightarrow\phi\right)\right),h(\delta)\in\{1,h(\phi)\}$.} and $\phi$ at a~lesser value.

It remains to show that $h$ satisfies $\mathbf{C1}$--$\mathbf{C4}$. Indeed, $\mathbf{C1}$ and $\mathbf{C2}$ hold because the premises of \eqref{equ:3.5.B} are sent to $1$. Furthermore, by the same reason, we have that $h\left(\bigvee\lozenge^{<1}_u\right)<1$ ($\mathbf{C3}$). Finally, since $h(\phi)<h(\delta)=1$, we obtain $\mathbf{C4}$ via an application of Proposition~\ref{prop:boxwrappingremark}.
\end{proof}

\begin{remark}\label{rem:whatisnewBox}
Let us return to the proof. Observe that it is crucial to use $\triangle\delta$ and $\triangle({}^*\Box^{=1}_u)$ and not $\delta$ and ${}^*\Box^{=1}_u$ in the premises of~\eqref{equ:3.5.A}. Indeed, if use the $\triangle$-less versions, then
\begin{align*}
\Th(\HKbiG),{}^*\Box^{=1}_u,\delta\not\models_{\biG}\left(\phi\rightarrow\bigvee\lozenge^{<1}_u\right)\rightarrow\bigvee\lozenge^{<1}_u
\end{align*}
does not guarantee the existence of $h$ s.t.\ $h(\delta)=1$ (which is necessary to establish $\mathbf{C4}$) and $h(\psi)=1$ ($\mathbf{C2}$) for every $\psi\in{}^*\Box^{=1}_u$.
\end{remark}

Next, we prove the counterpart of Proposition~4.7 in~\cite{RodriguezVidal2021}. Again, we will not be able to mimic it step-by-step since $\biG$ is not definable as the preservation of $1$. On the other hand, we need to fail the entailment in such a way that the premises are evaluated at $1$. Thus, we need a stronger version of~\cite[Lemma~4.6]{RodriguezVidal2021}. Our next proposition serves exactly this goal.
\begin{proposition}\label{prop:diamondwrapping}
Let $\phi\in\lozenge^{=\alpha}_u$ for some $\alpha>0$ and set
\begin{align*}
\delta'&\coloneqq\left(\phi\rightarrow\bigvee\lozenge^{<\alpha}_u\right)\rightarrow\bigvee\lozenge^{<\alpha}_u
\end{align*}
Then $u(\lozenge\triangle\delta')=1$.
\end{proposition}
\begin{proof}
Note that $u(\lozenge\phi)\!>\!u\left(\bigvee\lozenge\lozenge^{<\alpha}_u\right)$ by definition. Thus we have $u\left(\lozenge\phi\!\rightarrow\!\lozenge\bigvee\lozenge^{<\alpha}_u\right)\!<\!1$ by $\mathbf{K}$, whence $u\left({\sim}\triangle\left(\lozenge\phi\!\rightarrow\!\lozenge\bigvee\lozenge^{<\alpha}_u\right)\right)\!=\!1$. Now, we use ${\sim}\triangle\lozenge$ axiom to obtain that $u(\lozenge{\sim}\triangle(\phi\rightarrow\bigvee\lozenge^{<\alpha}_u))=1$. From here, since $\HGtriangle\vdash{\sim}\triangle(\phi\rightarrow\bigvee\lozenge^{<\alpha}_u)\rightarrow\triangle((\phi\rightarrow\bigvee\lozenge^{<\alpha}_u)\rightarrow\bigvee\lozenge^{<\alpha}_u)$, we obtain
$u(\lozenge\triangle\delta')\!=\!1$ by an application of $\mathbf{nec}$ and $\mathbf{K}$, as required.
\end{proof}

We are now ready to prove the $\lozenge$ counterpart of Proposition~\ref{prop:homomorphismexistencebox}.
\begin{proposition}\label{prop:homomorphismexistencediamond}
For any $\alpha>0$ and $\phi\in\lozenge^{=\alpha}_u$, there exists a~propositional homomorphism $h:\bimodalLtriangle\rightarrow[0,1]_{\biG}$, s.t.:
\begin{description}
\item[C1:] $h(\chi)=1$ for any $\chi\in\Th(\HKbiG)$;
\item[C2:] $h(\psi)=1$ for every $\psi\in{}^*\Box^{=1}_u$;
\item[C3:] $h(\rho)<1$ for every $\rho\in\lozenge^{<1}_u$;
\item[$\mathbf{C4}'$:] $h(\phi)>h(\sigma)$ for every $\sigma\in\lozenge^{<\alpha}_u$.
\end{description}
\end{proposition}
\begin{proof}
We assume
\begin{align}
\Th(\HKbiG),\triangle({}^*\Box^{=1}_u),\triangle\delta'\models_{\biG}\bigvee\lozenge^{<1}_u
\label{equ:3.7}
\end{align}
and reason for contradiction. Again, we use the completeness of $\HGtriangle$, Proposition~\ref{prop:propositionalreduction}, $\mathbf{K}$, $\MBox$, $\mathsf{P}$, and $\mathbf{Cr}$ to arrive at
\begin{align*}
\Th(\HKbiG),\triangle\Box({}^*\Box^{=1}_u)\models_{\biG}\lozenge\triangle\delta'\rightarrow\bigvee\lozenge\lozenge^{<1}_u
\end{align*}
It is easy to see that $u$ refutes this entailment: $u(\Th(\HKbiG))=1$ since $u\in W^\tau$, $u(\Box^*\Box^{=1}_u)=1$ by definition, whence $u(\triangle\Box^*\Box^{=1}_u)=1$, and $u(\lozenge\triangle\delta')=1$ by Proposition~\ref{prop:diamondwrapping} but $u\left(\bigvee\lozenge\lozenge^{<1}_u\right)<1$ by definition.

Thus, there exists a homomorphism $h$ that sends the premises of~\eqref{equ:3.7} to $1$\footnote{Again, observe that $\Th(\HKbiG)$ is closed under $\triangle$ and all other premises have $\triangle$ as their main connective.} and its conclusion to a~lesser value. Hence, $h$ satisfies $\mathbf{C1}$ and $\mathbf{C2}$. Furthermore, $h\left(\bigvee\lozenge^{<1}_u\right)<1$, and thus, $\mathbf{C3}$ is satisfied. Finally, since $h(\triangle\delta')=1$, we have that $h(\delta')=1$, whence $\mathbf{C4}'$ is satisfied as well.
\end{proof}
\begin{remark}\label{rem:whatisnewlozenge}
Again, the considerations of Remark~\ref{rem:whatisnewBox} apply for Proposition~\ref{prop:homomorphismexistencebox} as well. Since we need $h(\delta')=1$ to establish $\mathbf{C4}'$, we cannot use $\delta'$ by itself in the premise of~\eqref{equ:3.7} as the failure of the entailment does not (in general) guarantee that the premises are evaluated at $1$.
\end{remark}
\begin{remark}\label{rem:homomorphismfurtherproperties}
It is clear that since $h(\Box\mathbf{1})=1$, every homomorphism $h$ satisfying the conditions of Proposition~\ref{prop:homomorphismexistencebox} satisfies
\begin{description}
\item[C4.1:] $h(\phi)<1$.
\end{description}
Furthermore, from $h(\lozenge\mathbf{0})=0$, it follows that for every $h$ that satisfies the conditions of Proposition~\ref{prop:homomorphismexistencediamond}, it holds that
\begin{description}
\item[$\mathbf{C4'.1}$:] $h(\phi)>0.$
\end{description}
Finally, if $\mathbf{C1}$--$\mathbf{C3}$ are true for $h$, then the following properties hold for all $\theta,\theta'\in\bimodalLtriangle$.
\begin{description}
\item[$\mathbf{C2.a}$:] If $u(\lozenge\theta)\leq u(\Box\theta')$, then $h(\theta)\leq h(\theta')$ since $\theta\rightarrow\theta'\in{}^*\Box^{=1}_u$ using $\mathbf{FS}$.
\item[$\mathbf{C2.b}$:] If $\theta\in\Sfconst(\tau)$ and $u(\lozenge\theta)<u(\Box\theta')$, then $h(\theta)<h(\theta')$. For
\[\HKbiG\vdash((\Box\theta'\rightarrow\lozenge\theta)\rightarrow\lozenge\theta)\rightarrow(\Box((\theta'\rightarrow\theta)\rightarrow\theta)\vee\lozenge\theta)\]
$u(\lozenge\theta)<1$, and $u(\Box\theta'\rightarrow\lozenge\theta)\rightarrow\lozenge\theta)=1$ imply that $(\theta'\rightarrow\theta)\rightarrow\theta\in{}^*\Box^{=1}_u$ and $\mathbf{C3}$ implies $h(\theta)<1$.
\item[$\mathbf{C2.c}$:] If $u(\Box\theta)>0$, then $h(\theta)>0$.
\item[$\mathbf{C2.d}$:] If $u(\lozenge\theta)=0$, then $h(\theta)=0$.
\end{description}
\end{remark}

We can now establish the next statement which is analogous to propositions~4.5 and~4.8 in~\cite{RodriguezVidal2021} using propositions~\ref{prop:homomorphismexistencebox} and~\ref{prop:homomorphismexistencediamond} as well as remark~\ref{rem:homomorphismfurtherproperties}. The proof is exactly the same as in the original version.
\begin{proposition}\label{prop:limits}~
\begin{enumerate}
\item For any $\phi\in\Box^{=\alpha}_u$, $\alpha<1$, and $\varepsilon>0$ there is $u'\in W^\tau$ s.t.\ $uR^\tau u'$ and $u'(\phi)\in[\alpha,\alpha+\varepsilon]$.
\item For any $\phi\in\lozenge^{=\alpha}_u$, $\alpha>0$, and $\varepsilon>0$ there is $u'\in W^\tau$ s.t.\ $uR^\tau u'$ and $u'(\phi)\in[\alpha-\varepsilon,\alpha]$.
\end{enumerate}
\end{proposition}
The truth lemma can be established using Proposition~\ref{prop:limits} that guarantees that for every value $\alpha$ of $\Box\phi$ or $\lozenge\phi$, one can find an accessible state where the value of $\phi$ is arbitrarily close to $\alpha$. Thus, $\Box\phi$ will be indeed evaluated as the infimum and $\lozenge\phi$ as the supremum of $\phi$'s values in the accessible states. Again, the proof can be conducted in the same manner as in~\cite{RodriguezVidal2021}.
\begin{proposition}[Truth lemma]\label{prop:truthlemma}
For any $\phi\in\Sfconst(\tau)$, it holds that $e^\tau(\phi,u)=u(\phi)$.
\end{proposition}

Now, weak completeness will follow from the truth lemma and the validity of axioms and rules.
\begin{theorem}\label{theorem:KbiGcompleteness}
$\HKbiG$ is weakly complete: for any $\phi\in\bimodalLtriangle$, it holds that $\models_{\KbiG}\phi$ iff $\HKbiG\vdash\phi$.
\end{theorem}

The strong completeness is a bit more complicated.
\begin{theorem}\label{theorem:KbiGstrongcompleteness}
$\HKbiG$ is strongly complete: for any $\Gamma\cup\{\phi\}\subseteq\bimodalLtriangle$, it holds that $\Gamma\models_{\KbiG}\phi$ iff $\Gamma\vdash_{\HKbiG}\phi$.
\end{theorem}
\begin{proof}
The proof follows~\cite[Corollary~4.12]{RodriguezVidal2021}. The only two differences are that we need to account for $\triangle$ and that the $\KbiG$ entailment $\Gamma\models_{\KbiG}\chi$ is defined via the order on $[0,1]$. That is, if the entailment is refuted by $e$, then $\inf\{e(\phi,w):\phi\in\Gamma\}>e(\chi,w)$ for some $w\in\mathfrak{F}$. This, in turn, is equivalent to
\[\exists d\!\in\!(0,1]~\forall\phi\in\Gamma:e(\phi,w)\geq d\text{ but }e(\chi,w)<d\]

Now let $\Gamma\cup\{\phi\}\subseteq\bimodalLtriangle$ and $\Gamma\nvdash_{\HKbiG}\phi$. We consider the classical first order theory $\Gamma^*$ whose signature contains two unary predicates $W$ and $P$, one binary predicate $<$, binary functions $\circ$ and $\mathsf{s}$, unary function $\blacktriangle$, constants $0$, $1$, $c$, $d$, and a function symbol $f_\theta$ for each $\theta\in\bimodalLtriangle$. The axioms are as follows.
\begin{itemize}
\item $\forall x{\sim}(W(x)\wedge P(x))$
\item $\forall x(W(x)\vee{\sim}W(x))$
\item $P(d)$
\item ‘$\langle P,<\rangle$ is a strict linear order with $0$ and $1$ being its minimum and maximum and $0<d\leq1$’.
\item $\forall x\forall y((W(x)\wedge W(y))\rightarrow(\mathsf{s}(x,y)=1\vee\mathsf{s}(x,y)=0))$
\item $\forall x\forall y(P(x)\wedge P(y))\rightarrow((x\leq y\wedge x\circ y=1)\vee(x>y\wedge x\circ y=y))$
\item $\forall x(P(x)\rightarrow(x=1\wedge\blacktriangle(x)=1)\vee(x<1\wedge\blacktriangle(x)=0))$
\item For each $\theta,\theta'\in\bimodalLtriangle$, we add the following formulas.
\begin{itemize}
\item $\forall x(W(x)\rightarrow P(f_\theta(x)))$
\item $\forall x(W(x)\rightarrow f_{{\sim}\theta}(x)=(f_\theta(x)\circ 0))$
\item $\forall x(W(x)\rightarrow f_{\triangle\theta}(x)=\blacktriangle
(f_\theta(x)))$
\item $\forall x(W(x)\rightarrow f_{\theta\wedge\theta'}(x)=\min\{
f_\theta(x),f_{\theta'}(x)\})$
\item $\forall x(W(x)\rightarrow f_{\theta\vee\theta'}(x)=\max\{
f_\theta(x),f_{\theta'}(x)\})$
\item $\forall x(W(x)\rightarrow f_{\theta\rightarrow\theta'}(x)=
f_\theta(x)\circ f_{\theta'}(x))$
\item $\forall x(W(x)\rightarrow f_{\Box\theta}(x)=\inf\limits_{y}\{\mathsf{s}(x,y)\circ f_\theta(y)\})$
\item $\forall x(W(x)\rightarrow f_{\lozenge\theta}(x)=\sup\limits_{y}\{\min\{\mathsf{s}(x,y),f_\theta(y)\}\})$
\end{itemize}
\item For each $\gamma\in\Gamma$, we add $f_\gamma(c)\geq d$.
\item We also add $W(c)\wedge(f_\phi(c)<d)$.
\end{itemize}

The rest of the proof is identical to that in~\cite{RodriguezVidal2021}. For each finite subset $\Gamma^-$ of $\Gamma^*$, we let $\bimodalLtriangle^-=\{\theta:f_\theta\text{ occurs in }\Gamma^-\}$. Since $\bimodalLtriangle^-\cap\Gamma\nvdash_{\HKbiG}\phi$ by assumption, Theorem~\ref{theorem:KbiGcompleteness} entails that there is a~crisp pointed model $\langle\mathfrak{M},c\rangle$ with $\mathfrak{M}=\langle W,\mathsf{s}^{\Gamma^-},e^{\Gamma^-}\rangle$ being such that $e^{\Gamma^-}(\phi,c)<d$ and $e^{\Gamma^-}(\theta,c)\geq d$ for every $\theta\in\Gamma\cap\Gamma^-$. Thus, the following structure
\[\langle W\sqcup[0,1],W,[0,1],<,0,1,c,d,\circ,\blacktriangle,\mathsf{s}^{\Gamma^-},\{f_\theta\}_{\theta\in\bimodalLtriangle}\rangle\]
is a model of $\Gamma^-$. Now, by compactness and the downward L\"{o}wenheim--Skolem theorem, $\Gamma^*$ has a~countable model
\[\mathfrak{M}^*=\langle B,W,P,<,0,1,c,d,\circ,\blacktriangle,\mathsf{s}\{f_\theta\}_{\theta\in\bimodalLtriangle}\rangle\]
Now, we can embed $\langle P,<\rangle$ into $\langle\mathbb{Q}\cap[0,1],<\rangle$ preserving $0$, $1$ as well as all infima and suprema. Hence, we may w.l.o.g.\ assume that $\mathsf{s}$ is crisp and the ranges of $f_\theta$'s are contained in $[0,1]$. Then, it is straightforward to verify that $\mathfrak{M}=\langle W,S,e\rangle$, where $e(\theta,w)=f_\theta(w)$ for all $w\in W$ and $\theta\in\bimodalLtriangle$, is a crisp $\KbiG$ model with a distinguished world $c$ such that $v[\Gamma,c]\geq d$ and $e(\phi,c)<d$ for some $0<d\leq1$. Hence, $\inf\{e(\gamma,c):\gamma\in\Gamma\}>e(\phi,c)$, and thus, $\Gamma\not\models_{\KbiG}\phi$.
\end{proof}

We end the section by providing a complete calculus for $\KGsquare$ in $\bimodalLtrianglesquare$.
\begin{definition}[$\HKGsquare$ --- Hilbert-style calculus for $\KGsquare$]
$\HKGsquare$ contains the following axioms and rules.
\begin{description}
\item[A0:] All instances of $\HKbiG$ rules and axioms in $\bimodalLtrianglesquare$ language.
\item[$\mathsf{neg}$:] $\neg\neg\phi\leftrightarrow\phi$
\item[$\mathsf{DeM}\wedge$:] $\neg(\phi\wedge\chi)\leftrightarrow(\neg\phi\vee\neg\chi)$
\item[$\mathsf{DeM}\vee$:] $\neg(\phi\vee\chi)\leftrightarrow(\neg\phi\wedge\neg\chi)$
\item[$\mathsf{DeM}\!\rightarrow$:] $\neg(\phi\rightarrow\chi)\leftrightarrow(\neg\chi\wedge{\sim}\triangle(\neg\chi\rightarrow\neg\phi))$
\item[$\mathsf{DeM}\triangle$:] $\neg\triangle\phi\leftrightarrow{\sim\sim}\neg\phi$
\item[$\mathsf{DeM}{\sim}$:] $\neg{\sim}\phi\leftrightarrow{\sim}\triangle\neg\phi$
\item[$\mathsf{DeM}\Box\lozenge$:] $\Box\phi\leftrightarrow\neg\lozenge\neg\phi$
\end{description}
\end{definition}

It is instructive to observe that several axioms of $\HKbiG$ become redundant in $\HKGsquare$ if we assume contraposition \emph{as a rule} applied to theorems.
\begin{proposition}\label{prop:KGsquareredundancies}
The following $\HKbiG$ axioms are redundant in $\HKGsquare$ with contraposition:
\begin{itemize}
\item ${\sim}\lozenge\mathbf{0}$;
\item $\lozenge(\phi\vee\chi)\rightarrow(\lozenge\phi\vee\lozenge\chi)$;
\item ${\sim}\triangle(\lozenge\phi\rightarrow\lozenge\chi)\rightarrow\lozenge{\sim}\triangle(\phi\rightarrow\chi)$;
\item $\triangle\Box\phi\rightarrow\Box\triangle\phi$.
\end{itemize}
\end{proposition}
\begin{proof}
We begin with ${\sim}\lozenge\mathbf{0}$. Since ${\sim}$ can be defined via $\rightarrow$, we need to prove $\lozenge\mathbf{0}\rightarrow\mathbf{0}$. By contraposition, we can prove $\neg\mathbf{0}\rightarrow\neg\lozenge\mathbf{0}$. By De Morgan laws, we transform this into $\neg\mathbf{0}\rightarrow\Box\neg\mathbf{0}$. Now recall that $\HGsquare\vdash\neg\mathbf{0}$, whence $\HKGsquare\vdash\neg\mathbf{0}\rightarrow\Box\neg\mathbf{0}$.

For $\lozenge(\phi\vee\chi)\rightarrow(\lozenge\phi\vee\lozenge\chi)$, note again that we can prove $\neg(\lozenge\phi\vee\lozenge\chi)\rightarrow\neg\lozenge(\phi\vee\chi)$ instead. But this is equivalent to $\Box(\neg\phi\wedge\neg\chi)\rightarrow(\Box\neg\phi\wedge\Box\neg\chi)$ which is provable in $\KG$.

For ${\sim}\triangle\lozenge$, we proceed as follows. First, note that $$\Box({\sim}\neg\chi\vee\triangle(\neg\chi\rightarrow\neg\phi))\rightarrow(\lozenge{\sim}\neg\chi\vee\triangle(\Box\neg\chi\rightarrow\Box\neg\phi))$$ can be proven via an application of $\vee$-commutativity to $\mathbf{Cr}$, the Barcan's formula, and $\mathbf{K}$.

From here, since, $\HKGc\vdash\lozenge{\sim}\neg\chi\rightarrow{\sim}\Box\neg\chi$, we obtain $$\Box({\sim}\neg\chi\vee\triangle(\neg\chi\rightarrow\neg\phi))\rightarrow({\sim}\Box\neg\chi\vee\triangle(\Box\neg\chi\rightarrow\Box\neg\phi))$$
Now, applying ${\sim\sim}\triangle\psi\leftrightarrow\triangle\psi$, we obtain
$$\Box({\sim}\neg\chi\vee{\sim\sim}\triangle(\neg\chi\rightarrow\neg\phi))\rightarrow({\sim}\Box\neg\chi\vee{\sim\sim}\triangle(\Box\neg\chi\rightarrow\Box\neg\phi))$$
We use the De Morgan law for $\sim$ --- ${\sim}(\psi\wedge\psi')\leftrightarrow({\sim}\psi\vee{\sim}\psi')$ to get
$$\Box{\sim}(\neg\chi\wedge{\sim}\triangle(\neg\chi\rightarrow\neg\phi))\rightarrow{\sim}(\Box\neg\chi\wedge{\sim}\triangle(\Box\neg\chi\rightarrow\Box\neg\phi))$$
At this point, we apply the De Morgan laws for $\rightarrow$ and $\triangle$ and $\lozenge$ and $\Box$ which give us
$$\Box{\sim}\neg(\phi\rightarrow\chi)\rightarrow{\sim}\neg(\lozenge\phi\rightarrow\lozenge\chi)$$
Recall now that $\HGtriangle\vdash{\sim}\psi\leftrightarrow{\sim\sim\sim}\psi$ and $\HGtriangle\vdash\triangle{\sim}\psi\leftrightarrow{\sim}\psi$. Thus, we have
$$\Box{\sim}\triangle{\sim\sim}\neg(\phi\rightarrow\chi)\rightarrow{\sim}\triangle{\sim\sim}\neg(\lozenge\phi\rightarrow\lozenge\chi)$$
The application of $\mathsf{DeM}\triangle$ gives us
$$\Box{\sim}\triangle\neg\triangle(\phi\rightarrow\chi)\rightarrow{\sim}\triangle\neg\triangle(\lozenge\phi\rightarrow\lozenge\chi)$$
We can now apply $\mathsf{DeM}{\sim}$ to obtain
$$\Box\neg{\sim}\triangle(\phi\rightarrow\chi)\neg\rightarrow\neg{\sim}\triangle(\lozenge\phi\rightarrow\lozenge\chi)$$
Finally, we use the $\neg$ contraposition $\mathsf{neg}$, and $\mathsf{DeM}\Box\lozenge$ to get
$${\sim}\triangle(\lozenge\phi\rightarrow\lozenge\chi)\rightarrow\lozenge{\sim}\triangle(\phi\rightarrow\chi)$$

At last, we can see that $\triangle\Box\phi\rightarrow\Box\triangle\phi$ can be transformed via contraposition and De Morgan laws into $\lozenge{\sim\sim}\neg\phi\rightarrow{\sim\sim}\lozenge\neg\phi$ which is provable in $\HKGc$.
\end{proof}

The following completeness theorem is a straightforward corollary of Theorem~\ref{theorem:KbiGcompleteness} and Proposition~\ref{prop:+translation} since every $\phi\in\bimodalLtrianglesquare$ can be transformed into its NNF using the axioms of $\HKGsquare$.
\begin{theorem}\label{theorem:HKGsquarecompleteness}
$\HKGsquare$ is strongly complete: for any $\Gamma\cup\{\phi\}\subseteq\bimodalLtrianglesquare$, it holds that $\Gamma\models_{\KGsquare}\phi$ iff $\Gamma\vdash_{\HKGsquare}\phi$.
\end{theorem}
\section[Model-theoretic properties]{Model-theoretic properties of $\KbiG$ and $\KGsquare$\label{sec:semantics}}
In this section, we further investigate the semantical properties of $\KGsquare$. However, to simplify the presentation, we will formulate most results in the language of $\KG$ (i.e., without $\triangle$ and $\neg$). Still, they are applicable to $\KbiG$ and $\KGsquare$ by virtue of Proposition~\ref{prop:conservativity}.
\subsection{Transferrable formulas}
As we have already discussed in the introduction, it is known (cf.~e.g.~\cite{CaicedoMetcalfeRodriguezRogger2013}) that $\KG^c$ can be embedded into $\mathbf{K}$ if we replace each variable $p$ with ${\sim\sim}p$ in the formulas. Furthermore, some formulas defining useful classes of frames do not require any translation at all. For example~\cite{RodriguezVidal2021}, the following formulas define the same classes of frames both in $\mathbf{K}$ and $\KG^c$.
\begin{align*}
\Box p\rightarrow p&&p\rightarrow\lozenge p\tag{reflexivity}\\
\Box p\rightarrow\Box\Box p&&\lozenge\lozenge p\rightarrow\lozenge p\tag{transitivity}\\
p\rightarrow\Box\lozenge p&&\lozenge\Box p\rightarrow p\tag{symmetry}\\
\lozenge p\rightarrow\Box\lozenge p&&\lozenge\Box p\rightarrow\Box p\tag{Euclideanness}\\
\lozenge\mathbf{1}\tag{seriality}
\end{align*}

One should observe, however, that since $\Box$ and $\lozenge$ are not interdefinable in $\KG^c$~\cite[Lemma~6.1]{RodriguezVidal2021}, nor in $\KbiG$~\cite[Corollary~2]{BilkovaFrittellaKozhemiachenko2022IJCAR}, one needs \emph{both} formulas to define a class of frames in the bi-modal languages. On the other hand, the interdefinability of $\Box$ and $\lozenge$ in $\KGsquare$ allows for the use of only one of these formulas.

A natural question now is whether \emph{every} classical definition of a class of frames $\mathbb{F}$ defines $\mathbb{F}$ in $\KG^c$. Evidently, the answer is negative. For consider $\lozenge(p\vee{\sim}p)$. Even though it defines serial frames in classical modal logic, it does not do so in the G\"{o}del modal logic. In fact, $\lozenge(p\vee{\sim}p)$ can be refuted \emph{on every frame}.

One could also think that every $\phi$ that classically defines $\mathbb{F}$, defines it in $\KG^c$ as long as $\phi^-$ ($\phi$ with all modalities removed) is a $\mathsf{G}$-tautology. This turns out to be false too. For consider $\lozenge(\Box{\sim\sim}p\rightarrow{\sim\sim}\Box p)$. Clearly, $\phi^-$ is a $\mathsf{G}$-tautology. Classically, $\phi$ defines serial frames. However, it is not valid on the following frame:

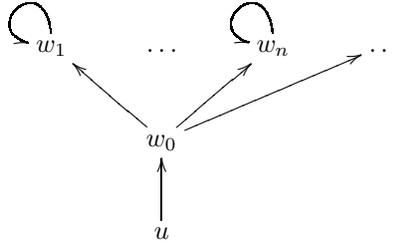
\vspace{.5em}

\[\xymatrix{w_1\ar@(u,l)&\ldots&w_n\ar@(u,l)&\ldots\\&w_0\ar[ul]\ar[ur]\ar[urr]&&\\&u\ar[u]&&}\]
Indeed, it suffices to put $e(p,w_i)=\frac{1}{i}$. Then, $e(\Box p,w_0)=0$, while $e(\Box{\sim\sim}p,w_0)=1$. Hence, $e(\Box{\sim\sim}p\rightarrow{\sim\sim}\Box p,w_0)=0$ and $e(\phi,u)=0$, although the frame is serial.

A question thus arises: which classes of formulas are \emph{transferrable}, i.e., define the same frames in $\mathbf{K}$ and $\KG^c$. In this section, we establish several such classes.
\begin{definition}[Transferrable formulas]\label{def:transferrableformula}
$\phi\in\{\mathbf{0},\mathbf{1},{\sim},\wedge,\vee,\rightarrow,\Box,\lozenge\}$ is called \emph{transferrable} iff for any crisp frame $\mathfrak{F}$~and $w\in\mathfrak{F}$, it holds that $\mathfrak{F},w\models_{\mathbf{K}}\phi$ iff $\mathfrak{F},w\models_{\KbiG}\phi$.
\end{definition}
\begin{proposition}\label{prop:closedtransfer}
Every closed formula (i.e., built only from constants $\mathbf{0}$ and $\mathbf{1}$) $\phi$ is transferrable.
\end{proposition}
\begin{proof}
Immediately since closed formulas on crisp frames have values in $\{0,1\}$.
\end{proof}
\begin{theorem}\label{prop:transferclosure}
Let $\phi$, $\phi'$, and $\psi$ be transferrable. Let further, $\mathsf{Var}(\phi)\cap\mathsf{Var}(\psi)=\varnothing$. Then, $\phi\wedge\phi'$, $\phi\vee\psi$, and $\Box\phi$ are transferrable.
\end{theorem}
\begin{proof}
The case of $\phi\wedge\phi'$ is straightforward, so we will only consider $\phi\vee\psi$ and $\Box\phi$.

\fbox{$\phi\vee\psi$}
\begin{align*}
\mathfrak{F},w\not\models_{\KbiG}\phi\vee\psi&\text{ iff }\mathfrak{F},w\not\models_{\KbiG}\phi\text{ and }\mathfrak{F},w\not\models_{\KbiG}\psi\\
&\text{ iff }\mathfrak{F},w\not\models_{\mathbf{K}}\phi\text{ and }\mathfrak{F},w\not\models_{\mathbf{K}}\psi\tag{by assumption}\\
&\text{ iff }\mathfrak{F},w\not\models_{\mathbf{K}}\phi\vee\psi
\end{align*}

\fbox{$\Box\phi$}
\begin{align*}
\mathfrak{F},w\not\models_{\KbiG}\Box\phi&\text{ iff }\exists w':wRw'\text{ and }\mathfrak{F},w'\not\models_{\KbiG}\phi\\
&\text{ iff }\exists w':wRw'\text{ and }\mathfrak{F},w'\not\models_{\mathbf{K}}\phi\tag{by assumption}\\
&\text{ iff }\mathfrak{F},w\not\models_{\mathbf{K}}\Box\phi
\end{align*}
\end{proof}

To establish further transfer results, we will need the notions of \emph{positive} and \emph{monotone} formulas.
\begin{definition}\label{def:monotoneformulas}~
\begin{itemize}
\item $\phi\in\bimodalLtrianglesquare$ is called \emph{monotone} iff it is built over $\{\wedge,\vee,\Box,\lozenge,\mathbf{1},\mathbf{0}\}$.
\item A \emph{monotone} formula is called \emph{positive} iff it does not contain $\mathbf{1}$ and $\mathbf{0}$.
\end{itemize}
\end{definition}
\begin{lemma}\label{lemma:transferclassicalisation}
Let $\phi$ and $\phi'$ be monotone. Let further, $e(\phi,w)>e(\phi',w')=x'$. Define
\begin{align}
e^\mathsf{cl}(p,u)&=
\begin{cases}
1&\text{ iff }e(p,u)>x'\\
0&\text{ iff }otherwise
\end{cases}\label{equ:classicalisation}
\tag{vCL}
\end{align}
Then $e^\mathsf{cl}(\phi,w)=1$ and $e^\mathsf{cl}(\phi',w')=0$.
\end{lemma}
\begin{proof}
We proceed by induction on the total number of connectives in $\phi$ and $\phi'$. The basis case of $\phi$ and $\phi'$ being variables or constants is straightforward. The cases of propositional connectives are easy as well.

For $e(\lozenge\phi,w)>e(\phi',w')$, we proceed as follows.
\begin{align*}
e(\lozenge\phi,w)>e(\phi',w')&\text{ iff }\sup\{e(\phi,u):wRu\}>e(\phi',w')\\
&\text{ iff }\exists u:wRu\text{ and }e(\phi,u)>e(\phi',w')\\
&\text{ iff }\exists u:wRu\text{ and }e^\mathsf{cl}(\phi,u)=1\text{ and }e^\mathsf{cl}(\phi',w')=0\tag{by IH}\\
&\text{ iff }e^\mathsf{cl}(\lozenge\phi,w)=1\text{ and }e^\mathsf{cl}(\phi',w')=0
\end{align*}
Other cases of modalities can be tackled in a similar manner.
\end{proof}
\begin{theorem}\label{prop:transfermonotone}
Let $\phi$ and $\phi'$ be monotone. Then $\phi\rightarrow\phi'$ is transferrable.
\end{theorem}
\begin{proof}
Immediately from Lemma~\ref{lemma:transferclassicalisation}.
\end{proof}

The final transfer result we are going to discuss in this section is that Sahlqvist formulas are transferrable. We recall the definition from~\cite{BlackburndeRijkeVenema2010}.
\begin{definition}\label{def:Sahlqvistformulas}
A \emph{Sahlqvist implication} ($\mathsf{SI}$) is a formula $\phi\rightarrow\chi$ with
\[\mathsf{SI}\ni\phi\coloneqq l\in\Prop\cup\{{\sim}p:p\in\Prop\}\cup\{\underbrace{\Box\ldots\Box}_{k\text{ times}} p:p\in\Prop,k\in\mathbb{N}\}\mid\phi\wedge\phi\mid\phi\vee\phi\mid\lozenge\phi\]
and $\chi$ being \emph{positive}. \emph{Sahlqvist formulas} ($\mathsf{SF}$) are obtained using the following grammar:
\[\mathsf{SF}\ni\psi,\psi'\coloneqq\tau\in\mathsf{SI}\mid\psi\wedge\psi\mid\psi\vee\psi'~(\mathsf{Var}(\psi)\cap\mathsf{Var}(\psi')=\varnothing)\mid\Box\psi\]
\end{definition}
\begin{theorem}\label{prop:Sahlqvisttransfer}
Sahlqvist formulas are transferrable.
\end{theorem}
\begin{proof}
By Theorem~\ref{prop:transferclosure}, it suffices to prove the statement only for Sahlqvist implications.

Let $\phi\rightarrow\chi\in\mathsf{SI}$. Assume that $e(\phi,w)>x'$ and $e(\chi,w')=x'\neq1$.

We show by induction on the total number of connectives that
\begin{align*}
e(\phi,u)>x'&\Rightarrow e^\mathsf{cl}(\phi,u)=1\tag{$\phi$ as in definition~\ref{def:Sahlqvistformulas}}\\
e(\chi,u')\leqslant x'&\Rightarrow e^\mathsf{cl}(\chi,u')=0\tag{$\chi$ is positive}
\end{align*}
The basis case of variables and constants is straightforward.

\fbox{$e({\sim}p,w)=1$, $e(\chi,w')<1$}
\begin{align*}
e({\sim}p,w)=1\text{ and }e(\chi,w')<1&\Rightarrow e(p,w)=0\text{ and }e(\chi,w')=x'<1\\
&\Rightarrow e^\mathsf{cl}(p,w)=0\text{ and }e^\mathsf{cl}(\chi,w')=0\tag{by IH since $p$ is positive}\\
&\Rightarrow e^\mathsf{cl}({\sim}p,w)=1\text{ and }e^\mathsf{cl}(\chi,w')=0
\end{align*}

The cases of propositional connectives as well as $\phi=\underbrace{\Box\ldots\Box}_{k\text{ times}}p$ are easy as well.

\fbox{$e(\lozenge\phi',w)>x'$, $e(\chi,w')=x'<1$}
\begin{align*}
e(\lozenge\phi',w)>x'\text{ and }e(\chi,w')=x'<1&\Rightarrow\sup\{e(\phi',u):wRu\}>x'\text{ and }e(\chi,w')=x'\\
&\Rightarrow\exists u:wRu\text{ and }e(\phi',u)>x'\text{ and }e(\chi,w')=x'\\
&\Rightarrow\exists u:wRu\text{ and }e^\mathsf{cl}(\phi',u)=1\text{ and }e^\mathsf{cl}(\chi,w')=0\tag{by IH}\\
&\Rightarrow e^\mathsf{cl}(\lozenge\phi',w)=1\text{ and }e^\mathsf{cl}(\chi,w')=0
\end{align*}
\end{proof}

Note that the two classes of transferrable formulas in theorems~\ref{prop:transfermonotone} and~\ref{prop:Sahlqvisttransfer} do not coincide as there are Sahlqvist implications that are not monotone and there are implications of monotone formulas that are not Sahlqvist. Note, furthermore, that the above theorems \emph{do not} characterise the class of transferrable formulas completely: for example, we can show that the G\"{o}del-L\"{o}b formula $\Box(\Box p\rightarrow p)\rightarrow\Box p$ is transferrable, even though it is neither monotone, nor Sahlqvist, nor obtained from transferrable formulas via Theorem~\ref{prop:transferclosure}.
\begin{proposition}\label{prop:GL}
Let $\mathfrak{F}=\langle W,R\rangle$ be a crisp frame. Then, $\mathfrak{F},w\models_{\KbiG}\Box(\Box p\rightarrow p)\rightarrow\Box p$ iff $R$ is transitive and does not contain an infinite chain $wRw_0Rw_1Rw_2R\ldots$ originating from $w$ (i.e., conversely well-founded).
\end{proposition}
\begin{proof}
Since $\KbiG$ valuations preserve classical values, we only prove the ‘only if’ direction. Let $\mathfrak{F}=\langle W,R\rangle$ be a crisp frame s.t.\ $R$ is transitive and conversely well founded. We let
\[e(\Box(\Box p\rightarrow p),w)=x>0\]
for some $w\in\mathfrak{F}$ and $e$ on $\mathfrak{F}$. Then, for every $w'\in R(w)$, it holds either $e(p,w)\geq x$ or $e(\Box p,w')\leq e(p,w')$.

Recall that $R$ does not have infinite chains beginning from $w$. Thus, $e(p,w'')\geq x$ for every $w''\in R(w)$ s.t.\ $R(w'')=\varnothing$ because $e(\Box p,w'')=1$ for every such $w''$. Denote the set of these states with $W_0$.

In general, for every $n\in\mathbb{N}$, we define $W_{-n}$ to be the set of all $t\in R(w)$ s.t.\ the longest $R$-sequence originating from $t$ has $n$ members.

It is clear that $w\in W_{-(k+1)}$ for some $k\in\mathbb{N}$ and that $R(w)=\bigcup\limits^{k}_{i=0}W_{-i}$. We show by induction on $k$ that $e(p,u)\geq x$ for every $u\in\bigcup\limits^{k}_{i=0}W_{i-1}$. The basis case is already shown. Assume that the statement holds for some $l$. We show it for $l+1$ and reason for a contradiction. Let $u'\in W_{-(l+1)}$ and $e(p,u')<x$. But then, since $R$ is transitive and irreflexive, we have $e(\Box p,u')\geq x$ by the induction hypothesis. Hence, $e(\Box p\rightarrow p,u')<x$ and further, $e(\Box(\Box p\rightarrow p),w)<x$, contrary to the assumption.

Thus, $e(p,u)\geq x$ for every $u\in R(w)$. But then, $e(\Box p,w)\geq x$, as required.
\end{proof}
\subsection[Glivenko's theorem]{Glivenko's theorem and its relatives}
In this section, we study the fragments of $\KbiG$ and $\KGsquare$ that admit Glivenko's theorem~\cite{Glivenko1929} that we present in its semantical form.
\begin{theorem}
$\phi$ is a classical propositional tautology iff ${\sim\sim}\phi$ is a (super-)intuitionistically valid propositional formula.
\end{theorem}

Glivenko's theorem in non-intermediate propositional logics is well studied (cf.,~e.g.~\cite{Ono2009} and the literature referred to therein). It is also known~\cite{Kleene1952} that the theorem holds for the $\exists$ fragment of the first-order intuitionistic logic. Furthermore, versions of Glivenko's theorem for modal intuitionistic logics are studied in~\cite{Bezhanishvili2001}.

Considering $\KbiG$ and $\KbiG^\mathsf{f}$, we, first, notice that the unrestricted version of Glivenko's theorem (unsurprisingly) fails: ${\sim\sim}\Box(p\vee{\sim}p)$ is not $\KbiG$ valid. In fact, it is easy to see that it defines finitely branching\footnote{A \emph{crisp} frame $\langle W,R\rangle$ is finitely branching iff $R(w)$ is finite for every $w\in W$. A \emph{fuzzy} frame is finitely branching iff $R^+(w)$ is finite for every $w$.} frames.
\begin{proposition}\label{prop:glivenkocounterexamplefb}
A (crisp or fuzzy) frame $\mathfrak{F}$ is finitely branching iff $\mathfrak{F}\models_{\KbiG}{\sim}{\sim}\Box(p\vee{\sim}p)$.
\end{proposition}
\begin{proof}
We show only the fuzzy case as the crisp one can be proven in a similar manner.

Assume that $\mathfrak{F}$ is finitely branching. Then, clearly, $e(\Box(p\vee{\sim}p),w)=\max\{R(w,w')\rightarrow_\mathsf{G}e(p\vee{\sim}p,w'):w'\in W\}>0$. Hence, $e({\sim}{\sim}\Box(p\vee{\sim}p),w)=1$.

Now let $\mathfrak{F}$ be infinitely branching, let $X\subseteq R^+(w)$ be countable and w.l.o.g.\ $R(w,w_i)\geqslant R(w,w_j)$ iff $i<j$ for every $w_i,w_j\in X$. We define $e(p,w_1)=\frac{R(w,w_1)}{2}$ and
\begin{align*}
e(p,w_{i+1})&=
\begin{cases}
\dfrac{e(p,w_i)}{2}&\text{ iff }e(p,w_i)\leqslant R(w,w_{i+1})\\
\dfrac{R(w,w_{i+1})}{2}&\text{ otherwise}
\end{cases}
\end{align*}
It is clear that $e({\sim}p,w_i)=0$ and that $e(p\vee{\sim}p,w_i)=e(p,w_i)$ for every $w_i\in R(w)$.

Observe that $\inf\{R(w,w')\rightarrow_{\mathsf{G}}e(p,w'):w\in W\}=0$. Thus, $e(\Box(p\vee{\sim}p),w)=0$, and thus $e({\sim}{\sim}\Box(p\vee{\sim}p),w)=0$, as required.
\end{proof}

In what follows, we will show that Glivenko's theorem holds in all finitely branching frames, and that, conversely, if Glivenko's theorem holds for a~logic of a~class of frames $\mathbb{F}$, then $\mathbb{F}$ does not contain infinitely branching frames. For this, we require some preliminary definitions and statements.
\begin{definition}[Logic of $\mathbb{F}$]
Let $\mathbb{F}$ be a class of frames. A \emph{$\KbiG$ ($\KGsquare$) logic of $\mathbb{F}$} is a set $\mathsf{L}\subseteq\bimodalLtriangle$ ($\mathsf{L}\subseteq\bimodalLtrianglesquare$) s.t.\ $\mathfrak{F}\models_{\KbiG}\mathsf{L}$ ($\mathfrak{F}\models_{\KGsquare}\mathsf{L}$) for every $\mathfrak{F}\in\mathbb{F}$.
\end{definition}
\begin{definition}\label{def:01valuation}
For any model $\mathfrak{M}=\langle W,R,e\rangle$, define a model $\mathfrak{M}^\mathsf{cl}=\langle W,R^\mathsf{cl},e^\mathsf{cl}\rangle$ s.t.
\begin{align*}
wR^\mathsf{cl}w'&=\begin{cases}
1&\text{ iff }wRw'\neq0\\
0&\text{ iff }wRw'=0
\end{cases}
&e^\mathsf{cl}(p,w)&=
\begin{cases}
1&\text{ iff }e(p,w)\neq0\\
0&\text{ iff }e(p,w)=0
\end{cases}
\end{align*}
For any frame $\mathfrak{F}=\langle W,R\rangle$, we set $\mathfrak{F}^\mathsf{cl}=\langle W,R^\mathsf{cl}\rangle$.
\end{definition}
\begin{lemma}\label{lemma:0isclassical}
Let $\phi$ be a formula over $\{\mathbf{0},\wedge,\vee,\rightarrow,\Box,\lozenge\}$. Then for any finitely branching frame $\mathfrak{F}$ and for any $e$ on $\mathfrak{F}$, it holds that
\begin{align}
e^\mathsf{cl}(\phi,w)&=
\begin{cases}
1&\text{ iff }e(\phi,w)\neq0\\
0&\text{ iff }e(\phi,w)=0
\end{cases}
\end{align}
\end{lemma}
\begin{proof}
We prove by induction. The cases when $\phi=p$ or $\phi=\mathbf{0}$ are trivial.

\fbox{$\phi=\psi\wedge\psi'$}
\begin{align*}
e(\psi\wedge\psi',w)=0&\text{ iff }e(\psi,w)=0\text{ or }e(\psi',w)=0\\
&\text{ iff }e^\mathsf{cl}(\psi,w)=0\text{ or }e^\mathsf{cl}(\psi',w)=0\tag{by IH}\\
&\text{ iff }e^\mathsf{cl}(\psi\wedge\psi',w)=0
\end{align*}

\fbox{$\phi=\psi\vee\psi'$} is dual.

\fbox{$\phi=\psi\rightarrow\psi'$}
\begin{align*}
e(\psi\rightarrow\psi',w)=0&\text{ iff }e(\psi,w)\neq0\text{ and }e(\psi',w)=0\\
&\text{ iff }e^\mathsf{cl}(\psi,w)=1\text{ or }e^\mathsf{cl}(\psi',w)=0\tag{by IH}\\
&\text{ iff }e^\mathsf{cl}(\psi\rightarrow\psi',w)=0
\end{align*}
\begin{align*}
e(\psi\wedge\psi',w)\neq0&\text{ iff }e(\psi',w)\neq0\\
&\text{ iff }e^\mathsf{cl}(\psi',w)=1\tag{by IH}\\
&\text{ iff }e^\mathsf{cl}(\psi\wedge\psi',w)=1
\end{align*}

\fbox{$\phi=\Box\psi$}
\begin{align*}
e(\Box\psi,w)=0&\text{ iff }\exists w':wRw'>0\text{ and }e(\psi,w')=0\\
&\text{ iff }\exists w':wRw'=1\text{ and }e^\mathsf{cl}(\psi,w')=0\tag{by IH}\\
&\text{ iff }e^\mathsf{cl}(\Box\psi,w)=0\tag{by finite branching}
\end{align*}

\fbox{$\phi=\lozenge\psi$}
\begin{align*}
e(\lozenge\psi,w)\neq0&\text{ iff }\exists w':wRw'>0\wedge e(\psi,w')\neq0\\
&\text{ iff }\exists w':wRw'=1\Rightarrow e^\mathsf{cl}(\psi,w')=1\tag{by IH}\\
&\text{ iff }e^\mathsf{cl}(\lozenge\psi,w)=1
\end{align*}
\end{proof}

The following unsurprising statement is immediate.
\begin{proposition}
Let $\phi\in\{\mathbf{0},\wedge,\vee,\rightarrow,\lozenge\}$. Then
\begin{enumerate}
\item $\phi$ is $\mathbf{K}$ valid iff ${\sim\sim}\phi$ is $\KbiG$ valid iff ${\sim\sim}\phi$ is $\KbiG^\mathsf{f}$ valid;
\item $\mathfrak{F}\models_{\mathbf{K}}\phi$ iff $\mathfrak{F}\models_{\KbiG}\phi$ for every crisp $\mathfrak{F}$.
\end{enumerate}
\end{proposition}
\begin{proof}
Note that in the proof of Lemma~\ref{lemma:0isclassical}, we use the finite branching only in the $\Box$ case but $\phi$ is $\Box$-free.
\end{proof}
\begin{theorem}\label{theorem:crispfbGlivenko}~
\begin{enumerate}
\item Let $\phi$ be a formula over $\{\mathbf{0},\wedge,\vee,\rightarrow,\Box,\lozenge\}$. Then it is $\mathbf{K}$-valid iff ${\sim\sim}\phi$ is $\KbiG$-valid ($\KbiG^\mathsf{f}$) on all finitely branching frames.
\item Let $\mathbb{F}$ be a class of (fuzzy or crisp) frames and let $\mathsf{L}$ be the $\KbiG$ logic of $\mathbb{F}$. Then, $\{{\sim\sim}\phi:\phi\text{ is }\mathbb{F}\models_{\mathbf{K}}\phi\}\subseteq\mathsf{L}$ implies that every $\mathfrak{F}\in\mathbb{F}$ is finitely branching.
\end{enumerate}
\end{theorem}
\begin{proof}
We begin with 1. Clearly, if $\phi$ is not valid in $\mathbf{K}$, there is a finite branching frame where it is invalidated by a~classical valuation. But classical valuations are preserved in $\KbiG$.

For the converse, let ${\sim\sim}\phi$ be not $\KbiG$-valid on some finitely branching frame $\mathfrak{F}$. Then, there exist $w\in\mathfrak{F}$ and $e$ on $\mathfrak{F}$ s.t.\ $e({\sim\sim}\phi,w)\neq1$. But then, $e(\phi,w)=0$. Hence, by Lemma~\ref{lemma:0isclassical}, we have a~classical valuation $e^\mathsf{cl}$ on $\mathfrak{F}^\mathsf{cl}$ s.t. $e^\mathsf{cl}(\phi,w)=0$. The result follows.

Consider 2. We reason by contraposition. Assume that $\mathbb{F}$ contains some infinitely branching frame $\mathfrak{F}$. But then $\mathfrak{F}\not\models_{\KbiG}{\sim\sim}\Box(p\vee{\sim}p)$. Thus, $\{{\sim\sim}\phi:\phi\text{ is }\mathbf{K}\text{ valid on }\mathbb{F}\}\not\subseteq\mathsf{L}$ as required.
\end{proof}

By conservativity (Proposition~\ref{prop:conservativity}), the above result extends to $\KGsquare$. Moreover, we can obtain a~result similar to Theorem~\ref{theorem:crispfbGlivenko} but with adding $\neg{\sim}$ on top of formulas instead of ${\sim\sim}$. This can be considered as a counterpart of Glivenko's theorem for $\mathsf{I}_4\mathsf{C}_4$ and its extensions\footnote{Recall that $\Gsquare$ is a linear extension of $\mathsf{I}_4\mathsf{C}_4$.} for it holds for every $\phi$ over $\{\mathbf{0},\wedge,\vee,\rightarrow\}$ that $\phi$ is classically valid iff $\neg{\sim}\phi$ is valid in $\mathsf{I}_4\mathsf{C}_4$.
\begin{theorem}\label{theorem:KGsquareglivenko}~
\begin{enumerate}
\item Let $\phi$ be a formula over $\{\mathbf{0},\wedge,\vee,\rightarrow,\Box,\lozenge\}$. Then it is $\mathbf{K}$-valid iff $\neg{\sim}\phi$ is $\KGsquare$-valid on all finitely branching crisp frames.
\item Let $\mathbb{F}$ be a class of crisp frames, and let $\mathsf{L}$ be the $\KGsquare$ logic of $\mathbb{F}$. Then, $\{\neg{\sim}\phi:\mathbb{F}\models_{\mathbf{K}}\phi\}\subseteq\mathsf{L}$ implies that every $\mathfrak{F}\in\mathbb{F}$ is finitely branching.
\end{enumerate}
\end{theorem}
\begin{proof}
Consider 1. It is clear that no classically valid $\phi$ can have $e_1(\phi,w)=0$, nor $e_2(\phi,w)=1$. Otherwise, by Lemma~\ref{lemma:0isclassical} and Proposition~\ref{prop:+isenough}, there is a classical valuation $e^\mathsf{cl}$ s.t. $e^\mathsf{cl}(\phi,w)=(0,1)$. Thus, $e(\neg{\sim}\phi,w)=(1,0)$, as required.

For 2, assume that $\mathbb{F}$ contains an infinitely branching frame $\mathfrak{F}$. Let now $R(w)$ be infinite for some $w\in\mathfrak{F}$ and $\{w_i:i\!\geq\!1,i\!\in\!\mathbb{N}\}\subseteq R(w)$. We set $e(p,w_i)\!=\!\left(\frac{1}{i+1},1-\frac{1}{i}\right)$. It is easy to see that $e(\Box(p\!\vee\!{\sim}p),w)\!=\!(0,1)$, whence $e(\neg{\sim}\Box(p\!\vee\!{\sim}p),w)\!=\!(0,1)$, and thus, $\{\neg{\sim}\phi\!:\!\mathbb{F}\models_{\mathbf{K}}\phi\}\not\subseteq\mathsf{L}$.
\end{proof}
\section{Decidability and complexity\label{sec:complexity}}
In this section, we establish that, as expected, the satisfiability and validity\footnote{Satisfiability and falsifiability (non-validity) are reducible to each other using $\triangle$: $\phi$ is satisfiable (falsifiable) iff ${\sim}\triangle\phi$ is falsifiable (satisfiable).} of $\KbiG$ and $\KGsquare$ are $\mathsf{PSPACE}$ complete. We apply the approach proposed in~\cite{CaicedoMetcalfeRodriguezRogger2013,CaicedoMetcalfeRodriguezRogger2017}.

The next definition is a straightforward adaptation of~\cite{CaicedoMetcalfeRodriguezRogger2013} to $\KbiG$.
\begin{definition}[$\mathsf{F}$-models of $\KbiG$]
An $\mathsf{F}$-model is a tuple $\mathfrak{M}=\langle W,R,T,e\rangle$ with $\langle W,R,e\rangle$ being a~$\KbiG$ model and $T:W\rightarrow\mathcal{P}_{<\omega}([0,1])$ be s.t.\ $\{0,1\}\subseteq T(w)$ for all $w\in W$. $e$ is extended to the complex formulas as in $\KbiG$ in the cases of propositional connectives, and in the modal cases, as follows.
\begin{align*}
e(\Box\phi,w)&=\max\{x\in T(w):x\leq\inf\{e(\phi,w'):wRw'\}\}\\
e(\lozenge\phi,w)&=\min\{x\in T(w):x\geq\sup\{e(\phi,w'):wRw'\}\}
\end{align*}
\end{definition}
\begin{example}[A finite $\mathsf{F}$-model]
Recall that there are no finite $\KbiG$ countermodels for $\phi=\triangle\lozenge p\rightarrow\lozenge\triangle p$. It is, however, easy to provide a finite $\mathsf{F}$-model of $\phi$ (cf.~Fig.~\ref{fig:Fmodelexample}). Indeed, it is clear that $e(\phi,w)=0$.
\begin{figure}
\centering
\[\xymatrix{w:\ar[r]&w':p=\frac{1}{2}}\]
\caption{$T(w)=\{0,1\}$, $T(w')$ can be arbitrary.}
\label{fig:Fmodelexample}
\end{figure}
One sees that $e(p,w')=\frac{1}{2}$, whence $\inf\{e(p,w'):wRw'\}=\frac{1}{2}$ as well. But then the minimal $T(w)$ that is at least as great as $\frac{1}{2}$ is $1$. Thus, $e(\triangle\lozenge p,w)=1$. On the other hand, $e(\triangle p,w')=0$, whence, $e(\lozenge\triangle p,w)=0$.
\end{example}

The next lemma is a straightforward extension of~\cite[Theorem~1]{CaicedoMetcalfeRodriguezRogger2013} to $\KbiG$. The proof is essentially the same since we add only $\triangle$ to the language.
\begin{lemma}\label{lemma:FFMP}
$\phi$ is $\KbiG$ valid iff $\phi$ is true in all $\mathsf{F}$-models iff $\phi$ is true in all $\mathsf{F}$-models whose depth is $O(|\phi|)$ s.t.\ $|W|\leq(|\phi|+2)^{|\phi|}$ and $|T(w)|\leq|\phi|+2$ for all $w\in W$.
\end{lemma}

It is now clear that $\KbiG$ (and hence $\KGsquare$) are decidable. To establish their complexity, we can utilise the algorithm described in~\cite{CaicedoMetcalfeRodriguezRogger2017}. The algorithm will work for $\KbiG$ since its only difference from $\KG^c$ is $\triangle$ which is an extensional connective. Another alternative would be to expand the tableaux calculus for $\KG^c$ from~\cite{Rogger2016phd} with the rules for $\triangle$ and use it to construct the decision procedure. The following statement is now immediate.
\begin{theorem}\label{theorem:KbiGpspacecomplete}
The satisfiability of $\KbiG$ (and hence, $\KGsquare$) is $\mathsf{PSPACE}$-complete.
\end{theorem}

\section{Conclusion\label{sec:conclusion}}
In this paper, we axiomatised crisp modal expansions of the bi-G\"{o}del logic and the paraconsistent G\"{o}del logic $\Gsquare$ in the bi-modal language with $\triangle$. We also established their complexity and investigated their semantical properties. Namely, we showed that (among others) Sahlqvist formulas and implications of monotone formulas define the same classes of frames in $\mathbf{K}$, $\KbiG$, and $\KGsquare$. Moreover, we established that Glivenko's theorem holds in the $\KbiG$ ($\KGsquare$) logic of a class of frames $\mathbb{F}$ iff $\mathbb{F}$ contains only finitely branching frames.

In future work, we plan to further investigate modal logics arising from $\biG$ and $\Gsquare$. First of all, we plan to axiomatise fuzzy versions of $\KbiG$ and $\KGsquare$. And while the axiomatisation of $\KbiG^\mathsf{f}$ may happen to be relatively straightforward, this seems to be not the case with ${\KGsquare}^\mathsf{f}$. Indeed, recall the proof of Proposition~\ref{prop:KGsquareredundancies}. There, by means of $\KGsquare$, we reduced the $\triangle\Box$ definition of crisp frames to the formula valid in all frames. This means that the standard definitions of $\Box$ and $\lozenge$ in ${\KGsquare}^\mathsf{f}$ will produce the logic that \emph{does not} extend $\KbiG^\mathsf{f}$. We leave its axiomatisation for the future research.

Secondly, as in $\KGsquare$, we treat truth and falsity of statements independently, it makes sense to have not one but two accessibility relations on a~frame: $R^+$ and $R^-$ that designate the degree of trust the agent puts in the assertions and denials given by sources. This also makes sense in the analysis of statistical evidence: assume that $w$ is some test that gives many false positives but almost no false negatives, while $w'$ is another test that gives few false positives and many false negatives. Thus, one would tend to believe in positive results provided by $w$ less in the ones provided by $w'$ and vice versa for the negative results.

Third direction of further research would be to devise description logics expanding $\Gsquare$. Description G\"{o}del logics~\cite{BobilloDelgadoGomez-RamiroStraccia2009,BobilloDelgadoGomez-RamiroStraccia2012,BorgwardtDistelPenaloza2014} are useful in the representation of vague or uncertain data which is not possible in the classical ontologies. In fact~\cite{BaaderPenaloza2011,BorgwardtPenaloza2012}, G\"{o}del description logics are the only\footnote{Note that although generalised concept inclusion of \L{}ukasiewicz description logics, and hence, its global entailment is undecidable, the local entailment still is~\cite{Vidal2021}.} decidable fuzzy description logics. On the other hand, there is a considerable amount of work done on paraconsistent description logics, i.e., logics whose underlying propositional fragment is paraconsistent. For example, description logics expanding $\mathsf{N4}$ are presented in~\cite{OdintsovWansing2003}, description logics over Belnap--Dunn logics are studied in~\cite{MaHitzlerLin2007}, and the ones over Priest's logic of paradox in~\cite{ZhangLinWang2010}. Still, to the best of our knowledge, there is no work done on paraconsistent fuzzy logics which would enable one to non-trivially reason with data which is both vague and contradictory.
\bibliographystyle{plain}
\bibliography{references.bib}
\end{document}